\documentclass[11pt,a4paper,twoside]{article}
\title{Parameter Height bounds for the Zilber Pink conjecture for PEL types III and IV}
\author{Bijay Raj Bhatta\thanks{Corresponding author: 
Department of Mathematics, University of Manchester, 
Alan Turing Building, Manchester M13 9PY. 
Email: bijay.bhatta@manchester.ac.uk} \\
Department of Mathematics, University of Manchester \\
Alan Turing Building, Manchester M13 9PY}

\date{}
\usepackage{amsmath,amssymb}
\usepackage{graphicx}  
\usepackage{pdfpages}
\usepackage{hyperref}
\usepackage{epstopdf}
\usepackage{inputenc}
\usepackage{titlesec} 
\usepackage{aliascnt}
\usepackage{tikz}
\usepackage{tipa}
\usepackage{xtab}
\usepackage{setspace}

\newcommand{\bQ}{\mathbb{Q}}
\newcommand{\bR}{\mathbb{R}}

\newcommand{\bZ}{\mathbb{Z}}
\newcommand{\ZZ}{\mathbb{Z}}
\newcommand{\QQ}{\mathbb{Q}}
\newcommand{\RR}{\mathbb{R}}
\newcommand{\CC}{\mathbb{C}}

\newcommand{\gG}{\mathcal{G}}

\newcommand{\cA}{\mathcal{A}}

\newcommand{\cF}{\mathcal{F}}

\newcommand{\cM}{\mathcal{M}}

\newcommand{\fc}{\mathfrak{c}}

\newcommand{\Ag}{\mathcal{A}_g}

\newcommand{\Hg}{\mathcal{H}_g}

\newcommand{\gGL}{\mathbf{GL}}

\newcommand{\gSp}{\mathbf{Sp}}

\newcommand{\rH}{\mathrm{H}}
\newcommand{\rM}{\mathrm{M}}

\renewcommand{\Im}{\mathop{\mathrm{Im}}}

\newcommand{\der}{\mathrm{der}}
\newcommand{\id}{\mathrm{id}}

\newcommand{\norm}[1]{\left\lVert#1\right\rVert}

\newcounter{constant}

\newcommand{\createC}[1]{%
    \refstepcounter{constant}%
    \label{C:#1}%
}

\newcommand{\newC}[1]{%
        \refstepcounter{constant}%
        C_{\theconstant}%
        \label{C:#1}%
}

\newcommand{\refC}[1]{C_{\ref*{C:#1}}}
\DeclareMathOperator{\Aut}{Aut}

\DeclareMathOperator{\covol}{covol}

\DeclareMathOperator{\diag}{diag}
\DeclareMathOperator{\disc}{disc}
\DeclareMathOperator{\End}{End}

\DeclareMathOperator{\Tr}{Tr}

\DeclareMathOperator{\GSp}{GSp}

\DeclareMathOperator{\Lie}{Lie}

\DeclareMathOperator{\Nm}{Nm}
\DeclareMathOperator{\Nrd}{Nrd}

\DeclareMathOperator{\Sp}{Sp}

\DeclareMathOperator{\Stab}{Stab}
\DeclareMathOperator{\Trd}{Trd}

\renewcommand{\Im}{\mathop{\mathrm{Im}}}

\usetikzlibrary{positioning}
\titleformat{\subsection}[runin]
  {\normalfont \bfseries}{\thesubsection}{1em}{}
\titleformat{\subsubsection}[runin]
  {\normalfont}{\thesubsubsection}{1em}{}
\makeatletter
\newcommand*\bigcdot{\mathpalette\bigcdot@{.5}}
\newcommand*\bigcdot@[2]{\mathbin{\vcenter{\hbox{\scalebox{#2}{$\m@th#1\bullet$}}}}}
\makeatother  
\usepackage{multirow}
\usepackage{mathtools}
\usepackage{amsfonts, latexsym, mdwlist, amsthm,amscd}
\numberwithin{equation}{subsection}
\usepackage{pifont} 
\usepackage{enumerate} 
\usepackage[super]{nth} 
\usepackage{color}
\usepackage{mathrsfs}
\usepackage{geometry}
\usepackage{float}
\def\abs#1{\left\lvert#1\right\rvert}

\makeatletter
\newtheorem*{rep@theorem}{\rep@title}
\newcommand{\newreptheorem}[2]{%
\newenvironment{rep#1}[1]{%
 \def\rep@title{#2 \ref{##1}}%
 \begin{rep@theorem}}%
 {\end{rep@theorem}}}
 \usepackage[capitalise]{cleveref}
\makeatother

\newtheorem{theorem}{Theorem}[section]

\newaliascnt{lemma}{theorem}
\newtheorem{lemma}[lemma]{Lemma}
\aliascntresetthe{lemma}

\newaliascnt{conjecture}{theorem}
\newtheorem{conjecture}[conjecture]{Conjecture}
\aliascntresetthe{conjecture}

\newaliascnt{proposition}{theorem}
\newtheorem{proposition}[proposition]{Proposition}
\aliascntresetthe{proposition}

\newaliascnt{corollary}{theorem}
\newtheorem{corollary}[corollary]{Corollary}
\aliascntresetthe{corollary}

\newaliascnt{definition}{theorem}
\newtheorem{definition}[definition]{Definition}
\aliascntresetthe{definition}

\newaliascnt{remark}{theorem}
\newtheorem{remark}[remark]{Remark}
\aliascntresetthe{remark}

\newaliascnt{property}{theorem}

\aliascntresetthe{property}

\newaliascnt{properties}{theorem}

\aliascntresetthe{properties}

\newaliascnt{claim}{theorem}
\newtheorem{claim}[claim]{Claim}
\aliascntresetthe{claim}

\newaliascnt{condition}{theorem}

\aliascntresetthe{condition}

\newaliascnt{hypothesis}{theorem}

\aliascntresetthe{hypothesis}

\newtheorem*{lemma*}{Lemma}
\newtheorem*{proposition*}{Proposition}
\newtheorem*{theorem*}{Theorem}
\newtheorem*{corollary*}{Corollary}
\newtheorem*{claim*}{Claim}

\crefname{theorem}{Theorem}{Theorems}
\Crefname{theorem}{Theorem}{Theorems}

\crefname{lemma}{Lemma}{Lemmas}
\Crefname{lemma}{Lemma}{Lemmas}

\crefname{conjecture}{Conjecture}{Conjectures}
\Crefname{conjecture}{Conjecture}{Conjectures}

\crefname{proposition}{Proposition}{Propositions}
\Crefname{proposition}{Proposition}{Propositions}

\crefname{corollary}{Corollary}{Corollaries}
\Crefname{corollary}{Corollary}{Corollaries}

\crefname{definition}{Definition}{Definitions}
\Crefname{definition}{Definition}{Definitions}

\crefname{remark}{Remark}{Remarks}
\Crefname{remark}{Remark}{Remarks}

\crefname{property}{Property}{Properties}
\Crefname{property}{Property}{Properties}

\crefname{properties}{Properties}{Properties}
\Crefname{properties}{Properties}{Properties}

\crefname{claim}{Claim}{Claims}
\Crefname{claim}{Claim}{Claims}

\crefname{condition}{Condition}{Conditions}
\Crefname{condition}{Condition}{Conditions}

\crefname{hypothesis}{Hypothesis}{Hypotheses}
\Crefname{hypothesis}{Hypothesis}{Hypotheses}

\usepackage[numbers]{natbib}
\usepackage{url}
\begin{document}
\maketitle
\pagenumbering{arabic}

\begin{abstract}
   We prove the Zilber-Pink conjecture for the intersection of an irreducible Hodge generic algebraic subvariety $ V \subset \mathcal{A}_g$ with special subvarieties of all simple PEL types other than $\mathbb{Z}$, under the assumption of the Large Galois Orbits conjecture. In particular, we establish parameter height bounds for the arithmetic ingredients of the Pila-Zannier strategy in the case of Albert types~III and~IV. This paper is a sequel to \citep{Lattices}. 

\end{abstract}

\section{Introduction}

In this paper we prove the following theorem for $\mathcal{A}_g(\mathbb{C})$, the moduli space of principally polarised abelian varieties of dimension~$g$ (a standard example of a Shimura variety).

\begin{theorem} \label{zp1}
Let $g \ge 3$.  
Let $\Sigma$ be the set of points $s \in \mathcal{A}_g(\mathbb{C})$ such that the associated abelian variety $A_s$ has endomorphism algebra which is a division algebra of Albert type~III or~IV.  
Let $C$ be an irreducible Hodge generic algebraic curve in $\mathcal{A}_g$.

If $C$ satisfies \cref{GO0}, then the intersection $C \cap \Sigma$ is finite.
\end{theorem}

\begin{conjecture} \label{GO0}
Let $\Sigma \subset \mathcal{A}_g$ be as in \cref{zp1}.  
Let $C \subset \mathcal{A}_g$ be an irreducible Hodge generic algebraic curve defined over a finitely generated field $L \subset \mathbb{C}$.  
Then there exist positive constants $\newC{GO0-mult}$ and 
$\newC{GO0-exp}$, depending only on $g$, $L$ and $C$, such that for every point $s \in C \cap \Sigma$,
\[
\#\Aut(\mathbb{C}/L)\cdot s 
\ge
\refC{GO0-mult}\,
\bigl|\disc(\End(A_s))\bigr|^{\refC{GO0-exp}}.
\]
\end{conjecture}

A stronger version is in fact known: both \cref{zp1} and \citep[Theorem~1.3]{Lattices} extend to any irreducible Hodge generic subvariety $V \subseteq \mathcal{A}_g$ of dimension at least~$1$.  
This generalisation follows from work of Cassani \citep{Cas24}, assuming a parameter height bound and the Large Galois Orbits condition for $V$; see \citep[Section~6.2.3]{Daw25}.  
Because our parameter height bound does not depend on $\dim V$, the result applies to all such $V$, not only to curves.  
This yields the following theorem.

\begin{theorem}\label{emt}
Let $V$ be an irreducible Hodge generic subvariety of $\mathcal{A}_g$.  
Let $\Sigma$ be the union of special subvarieties $Z$ of PEL type whose generic abelian variety is simple and for which $\textnormal{codim}(Z)>\dim(V)$.

If $V$ satisfies \cref{GO0} (with $C$ replaced by $V$ and this choice of $\Sigma$), then $V \cap \Sigma$ is not Zariski dense in $V$.
\end{theorem}

Unless stated otherwise, all endomorphisms of abelian varieties in this paper are taken over an algebraically closed field.  
Recall that a \textbf{special subvariety} of a Shimura variety $S$ means an irreducible component of a Shimura subvariety, and a subvariety is \textbf{Hodge generic} if it is not contained in any proper special subvariety.

Theorem \ref{zp1} represents a specific case of the Zilber-Pink conjecture on unlikely intersections.  
For a discussion of how \cref{zp1} fits into this framework, see \citep[Section~1.C]{Lattices}.  

\begin{conjecture} \label{ZP2} \citep[Conjecture~1.3]{Pin05}
Let $S$ be a Shimura variety and let $V$ be an irreducible Hodge generic subvariety of~$S$.  
Then the intersection of $V$ with the special subvarieties of $S$ of codimension strictly greater than $\dim V$ is not Zariski dense in $V$.
\end{conjecture}

The Zilber-Pink conjecture remains one of the major open problems in arithmetic geometry.  
Pila and Zannier introduced a method combining o-minimality and Diophantine estimates, first applied to a new proof of the Manin-Mumford conjecture \citep{PZ08}.  
For Shimura varieties, this approach was developed by Daw and Renn \citep{DR18}, following ideas of Habegger and Pila \citep{HP16}.  
Two essential arithmetic ingredients of this strategy are a parameter height bound and the Large Galois Orbits (LGO) condition.

Albert classified division $\mathbb{Q}$-algebras with positive involution into four types \citep[Sec.~21, Thm.~2]{Mum74} (see Section~\ref{div}).  
Height bounds for types~I and~II were proved in \citep{Lattices}; extending these results to types~III and~IV gives a conditional proof of \cref{zp1}, and leads to unconditional results such as \cref{CorLGO} and \cref{cor2}.

We recall an unconditional result on LGO due to Daw and Orr.

\begin{theorem}\citep[Theorem 1.1]{DO23b}\label{LGOtheorem}
Let $g \ge 2$.  
Let $C$ be a geometrically irreducible Hodge generic curve in $\mathcal{A}_g$, defined over $\overline{\mathbb{Q}}$, whose Zariski closure in the Baily-Borel compactification meets the zero-dimensional boundary stratum.  
Then there exist constants $\newC{mom},\newC{mbe}>0$ such that for every $s \in C(\overline{\mathbb{Q}})$ for which $A_s$ has unlikely endomorphisms,
\[
[\mathbb{Q}(s):\mathbb{Q}]
\ge
\refC{mom}\,
\bigl|\disc(\End(A_s))\bigr|^{\refC{mbe}}.
\]
\end{theorem}

\begin{corollary}\label{CorLGO}
Let $g \ge 3$ and let $C$ be as above.  
Then $C(\overline{\mathbb{Q}})$ contains only finitely many points $s$ for which $A_s$ is simple and $\End(A_s)\neq\mathbb{Z}$.
\end{corollary}

\begin{proof}
This follows from \citep[Theorem~1.3]{Lattices}, \cref{zp1}, and \cref{LGOtheorem}.
\end{proof}

We also have an unconditional LGO result due to Fit\`e and Goodman \citep{FG25}.

\begin{theorem}\label{cor2}
Let $V$ be an irreducible Hodge generic subvariety of $\mathcal{A}_3$.  
Let $\Sigma$ be the union of PEL-type special subvarieties $Z$ for which the generic endomorphism ring is $\mathcal{O}_M$, where $M$ is an imaginary quadratic field, and satisfying $\textnormal{codim(Z)}>\dim(V)$.  
Then $V \cap \Sigma$ is not Zariski dense in $V$.
\end{theorem}

\begin{proof}
The statement follows from \cref{zp1} and \cref{emt}, together with \cref{impcor}, which in turn follows from \citep{FG25} (see \cref{FG}).
\end{proof}

A crucial tool is the height bound stated below and proved in \cref{sec:height bound}.  
This result is the analogue for types~III and~IV of \citep[Theorem~1.2]{Lattices}, originally proved for types~I and~II and used in \citep[Proposition~7.1]{Lattices}.


\begin{theorem} \label{height-bound}
Let $D$ be the type III or IV division $\mathbb{Q}$-algebra.
Let $\dag$ be a positive involution on ~$D$.
Let $V$ be a left $D$-vector space of dimension~$m$, equipped with a non-degenerate $(D,\dag)$-skew-Hermitian form $\psi \colon V \times V \to D$.
Let $L$ be a $\bZ$-lattice of full rank in~$V$ such that $\Trd_{D/\mathbb{Q}} \psi(L \times L) \subset \bZ$.
Let $R=\Stab_D(L)$ denote the stabiliser of $L$ in~$D$.

Then there exists a $D$-basis $v_1, \dotsc, v_m$ for $V$ such that:
 \begin{enumerate}[(i)]
\item $v_1, \dotsc, v_m \in L$;
\item the basis is weakly unitary with respect to $\psi$;
\item $[L:Rv_1 + \dotsb + Rv_m] \leq \newC{mmm} \abs{\disc(R)}^{\newC{beta}} \abs{\disc(L)}^{\newC{aaa}}$;
\item $\abs{\psi(v_i, v_j)}_D \leq \newC{bbb} \abs{\disc(R)}^{\newC{ccc}} \abs{\disc(L)}^{\newC{mml}}$ for $1 \leq i, j \leq m$.
\end{enumerate}
The constants $\refC{mmm}, \dotsc, \refC{mml}$ depend only on $m$ and $\dim_\mathbb{Q}(D)$.
\end{theorem}

Our proof of \cref{zp1} follows the same overall outline as \citep{Lattices}, in line with methods used in \citep{HP16, DR18, DO21a, DO21b}.  
The main new ingredient is the construction of a suitable parameter space $Y$ for special subvarieties of simple PEL type~III or~IV.  
The idea is to apply the Habegger-Pila-Wilkie counting theorem \citep[Cor.~7.2]{HP16} to a definable set
\[
D = \{(y,z) \in Y \times \mathcal{C} : z \in X_y\},
\]
where $Y$ is semi-algebraic, the sets $X_y$ belong to a fixed class of ``pre-special'' subvarieties for every rational point $y \in Y \cap \bQ^n$, and $\mathcal{C}$ is a definable subset of Siegel upper half-space. For each $s \in \Sigma_{d,e,m, \textbf{s}}$ (see \cref{FP1}), some lift $z$ lies in $X_y$ for a rational point $y \in Y(\mathbb{Q})$ whose height $H(y)$ is polynomially bounded in terms of $\mathrm{End}(A_s)$. The height bound \cref{height-bound} is then used to control rational points $y \in Y(\mathbb{Q})$.  
For details on this step, see \citep[Section~1.D]{Lattices}.

We now outline the strategy distinguishing the proof of \cref{zp1} from the approach in \citep{Lattices}. For notation and terminology, we refer to \citep{Lattices}.
Two major differences you can find in this paper compared to Daw-Orr paper is that we had to deal with signature issues in the type IV case as well as the highly non-commutative part of types III and IV compared to types I and II as there wasn't a direct generalisation of Lemma 2.11 of their paper, instead we use \cref{Rs} and \cref{w1w2} to deal with the issues of non-commutativity. We also use \citep[Lemma 3.15]{UY13} (\cref{conj-class-datum}) to show that for a given $(\mathcal{G},X)$ a Shimura Datum and $\mathcal{H} \subseteq \mathcal{G}$ a subgroup of $\mathcal{G}$, there exist finitely many $X_\mathcal{H}$ such that $(\mathcal{H}, X_\mathcal{H})$ is a Shimura Subdatum. Since conjugacy classes are disjoint from one another, the continuous path constructed in \ref{ctspath} lies in the unique pre-special subvariety of $X$ associated with $\mathcal{H}$, in one of the conjugacy classes of ${X_i}^+$. 

Let $\mathcal{G}=\GSp_{2g}$ and $(\mathcal{G},X^+)$ be the Shimura datum described in Section~\ref{Shimura-date}.  
By \cref{conjclass1}, the components $(\mathcal{H},X_\mathcal{H}^+)$ associated with simple PEL type~III/IV special subvarieties fall into finitely many $\mathcal{G}(\mathbb{R})$-conjugacy classes.  
Thus we can prove \cref{zp1} separately for each class.  
For type~III, the argument follows \citep{Lattices}.  
For type~IV, we combine conjugacy considerations with signature constraints, choosing positive integers $d,e,m,\mathbf{s}$ and taking $\mathcal{H}_0$ as in \eqref{eqn:H0}.  
We then prove \cref{zp1} for
$\Sigma_{d,e,m,\mathbf{s}}$, the union of type~IV special loci corresponding to groups $\mathcal{G}(\mathbb{R})$-conjugate to~$\mathcal{H}_0$.

\subsection{Structure of the Paper}
The paper is structured into two main sections (Sections 3 and 4). In Section 2, we introduce the necessary notation and lemmas that will be used in Sections 3 and 4. \cref{cases-of-ZP} is dedicated to proving our main theorem, \cref{zp1}. A key ingredient in our approach is the height bound, which is stated in \cref{height-bound} and proved in \cref{sec:height bound}. The challenges arising from the non-commutativity of types III and IV are addressed in \cref{antisydmo}, while the signature-related issues of type IV are handled in \cref{unitrary}.

\pagebreak
\subsection{Acknowledgements} The author wishes to express his sincere gratitude to his supervisor, Martin Orr, for his invaluable guidance and support throughout this work-without which this project would not have been possible. He is also grateful to Chris Daw for many insightful comments that have helped to improve earlier drafts of this paper. This work was supported by the Heilbronn Institute for Mathematical Research.

\section{Division algebras}\label{div}

This section collects definitions and notation from \citep{Lattices}. Let $D$ be a semisimple $k$-algebra.

\begin{definition}
An \textbf{involution} $\dag$ on $D$ is a $k$-linear map $D \to D$ such that $\dag \circ \dag = \id_D$ and $(ab)^\dag = b^\dag a^\dag$ for all $a,b \in D$.
\end{definition}

The involution $\dag$ is of the \emph{first kind} if it restricts to the identity on the centre $F$ of $D$; otherwise it is of the \emph{second kind}. This follows the convention of \citep{Lattices}.

For every $a \in D$ we have $\Trd_{D/k}(a^\dag) = \Trd_{D/k}(a)$. Hence the bilinear form
\[
(a,b) \longmapsto \Trd_{D/k}(ab^\dag)
\]
on $D$ is symmetric. The involution $\dag$ is called \textbf{positive} if this form is positive definite (equivalently, if the non-reduced trace form $(a,b) \mapsto \Tr_{D/k}(ab^\dag)$ is positive definite).

Given a semisimple $k$-algebra $(D,\dag)$ with positive involution, define a norm on $D_{\mathbb{R}}$ by
\[
\abs{a}_D = \sqrt{\Trd_{D_{\mathbb{R}}/\mathbb{R}}(aa^\dag)}.
\]

For background on semisimple algebras, reduced norms, and traces, see \citep[Sec.~9]{Rei75}.

\begin{lemma}[{\citep[Lemma~2.1]{Lattices}}]\label{lm}
Let $(D,\dag)$ be a semisimple $k$-algebra with positive involution. Then
\[
\abs{ab}_D \le \abs{a}_D \abs{b}_D \qquad \text{for all } a,b \in D_{\mathbb{R}}.
\]
\end{lemma}

Let $k \subset \mathbb{R}$, and let $(D,\dag)$ be a semisimple $k$-algebra with positive involution. Let $V$ be a left $D$-module. A function $\abs{\cdot} : V_{\mathbb{R}} \to \mathbb{R}$ is a \textbf{$D$-norm} if it arises from a positive definite inner product on $V_{\mathbb{R}}$ and satisfies
\[
\abs{av} \le \abs{a}_D \abs{v} \qquad \text{for all } a \in D_{\mathbb{R}},\ v \in V_{\mathbb{R}}.
\]
By Lemma~\ref{lm}, the norm $\abs{\cdot}_D$ is a $D$-norm on $D_{\mathbb{R}}$.

\pagebreak 

Division $\mathbb{Q}$-algebras $(D,\dag)$ with positive involution are classified according to the structure of $D_{\mathbb{R}}$ \citep[Sec.~21, Thm.~2]{Mum74}:

\begin{enumerate}
\item[\textbf{Type I.}]
$D = F$, a totally real number field, with trivial involution. Here $D_{\mathbb{R}} \cong \mathbb{R}^e$.

\item[\textbf{Type II.}]
$D$ is a non-split totally indefinite quaternion algebra over a totally real field $F$, so $D_{\mathbb{R}} \cong \mathrm{M}_2(\mathbb{R})^e$. The involution becomes matrix transpose on each factor.

\item[\textbf{Type III.}]
$D$ is a totally definite quaternion algebra over a totally real field $F$, hence $D_{\mathbb{R}} \cong \mathbb{H}^e$, where $\mathbb{H}$ is Hamilton's quaternions ($d=\sqrt{\dim_F D}=2$, $e=[F:k]$). The involution is the canonical one $a \mapsto \Trd_{D/F}(a)-a$.

\item[\textbf{Type IV.}]
$D$ is a division algebra whose centre is a CM field $F$. The involution restricts to complex conjugation on $F$. In this case
$D_{\mathbb{R}} \cong \mathrm{M}_d(\mathbb{C})^e$, where $d=\sqrt{\dim_F D}$ and $2e = [F:k]$.
\end{enumerate}

Let $k$ be any field and $(D,\dag)$ a semisimple $k$-algebra with involution. Let $V$ be a left $D$-module. A \textbf{$(D,\dag)$-skew-Hermitian form} on $V$ is a $k$-bilinear map $\psi : V \times V \to D$ satisfying
\[
\psi(y,x) = -\psi(x,y)^\dag, \qquad
\psi(ax,by) = a\psi(x,y)b^\dag
\]
for all $a,b \in D$ and $x,y \in V$. It is \textbf{non-degenerate} if for every $x \ne 0$ in $V$ there exists $y \in V$ with $\psi(x,y) \ne 0$.

A \textbf{$(D,\dag)$-compatible skew-symmetric form} on $V$ is a skew-symmetric $k$-bilinear map $\phi : V \times V \to k$ satisfying
\[
\phi(ax,y) = \phi(x, a^\dag y)
\]
for all $a \in D$ and $x,y \in V$. A pair $(V,\phi)$ with such a form is called a symplectic $(D,\dag)$-module in \citep[Sec.~8]{Mil05}.

\subsection{Preliminary Lemmas}

The next lemma is the analogue of \citep[Lemma~2.2]{Lattices} for division algebras of type~IV, with $F$ replaced by a maximal totally real subfield $F_0$. The same bounds hold for type~III algebras, even though in that case $F=F_0$ and $d=2$ and the algebra is not $\mathbb{R}$-split.

\begin{lemma}\label{Normbound}
Let $(D,\dagger)$ be a semisimple $k$-algebra with a positive involution of type~III or~IV. Let $F$ be its centre and $F_0$ a maximal totally real subfield of $F$. Then:
\begin{enumerate}[(i)]
\item (Type~III).
\[
\abs{\Nrd_{D_{\mathbb{R}}/F_{\mathbb{R}}}(a)}_D \le d^{(1-d)/2} \abs{a}_D^{\,d}, \qquad
\abs{\Nrd_{D_{\mathbb{R}}/\mathbb{R}}(a)} \le (de)^{-de/2} \abs{a}_D^{\,de}.
\]

\item (Type~IV).
\[
\abs{\Nrd_{D_{\mathbb{R}}/F_{0,\mathbb{R}}}(a)}_D \le d^{(1-2d)/2} \abs{a}_D^{\,2d}, \qquad
\abs{\Nrd_{D_{\mathbb{R}}/\mathbb{R}}(a)} \le (de)^{-de} \abs{a}_D^{\,2de}.
\]
\end{enumerate}
\end{lemma}

\begin{proof}
The argument is analogous to that of \citep[Lemma~2.2]{Lattices}.
\end{proof}
\begin{theorem}\label{PSP}
Given that $M$ is a positive definite Hermitian matrix and $N$ a Hermitian matrix, there exists a sufficiently large real number $t$ such that $N + tM$ is positive definite.
\end{theorem}

\begin{proof}
As both $M$ and $N$ are Hermitian, so their eigenvalues are real. Since $M$ is positive definite, all its eigenvalues are positive. Let the eigenvalues of $M$ be denoted by $\lambda_1, \lambda_2, \ldots, \lambda_n$ with $\lambda_i > 0$ for all $i$. 
Let the eigenvalues of $N$ be denoted by $\mu_1, \mu_2, \ldots, \mu_n$ with $\mu_i$ for all $i$.
We consider $N + tM$, for any $t \in \mathbb{R}$, the matrix $N + tM$ is Hermitian since it is the sum of two Hermitian matrices. 

For large $t$, the term $tM$ will dominate $N$. More formally, for any non-zero vector $x \in \mathbb{C}^n$:
   \[
   x^\dagger (N + tM) x = x^\dagger N x + t x^\dagger M x
   \]
   Since $M$ is positive definite, $x^\dagger M x > 0$ for all non-zero $x$, there exists a constant $c > 0$ such that $x^\dagger M x \geq c \abs{x}^2$. 

   For sufficiently large $t$, say $t > \frac{\max_i \abs{\mu_i}}{c}$, we have:
   \[
   x^\dagger (N + tM) x > 0 \quad \text{for all non-zero} \quad x \in \mathbb{C}^n
   \]
Since $x^\dagger (N + tM) x > 0$ for all non-zero $x \in \mathbb{C}^n$, the matrix $N + tM$ is positive definite for sufficiently large $t$.

\end{proof}

The following proposition is a standard result in linear algebra and is therefore stated without proof.
\begin{proposition}\label{2ideal}
     Let $V$ be a finite-dimensional right vector space over a division algebra $D_0$, and let $\psi: V \times V \to D_0$ be a non-degenerate skew-Hermitian form. Then the map
\[
\Phi: V \to \operatorname{Hom}(V, D_0), \quad x \mapsto \psi(x, -)
\]
is an isomorphism of right $D_0$-modules. Consequently, for any given $a \in D_0$, there exists $y \in V$ such that $\psi(x, y) = a$ for any fixed $x$ with $\psi(x, -)$ surjective.
\end{proposition}



\subsection{Unitary bases and signature} \label{unitrary}
Let $k$ be a field (in practice $k=\mathbb{Q} \textnormal{ or } \mathbb{R})$ and let $(D,\dag)$ be a semisimple $k$-algebra with an involution.
Let $V$ be a free left $D$-module and let $\psi : V \times V \to D$ be a $(D,\dag)$-skew-Hermitian form.

We say that a $D$-basis $v_1, \dotsc, v_m$ is \textbf{weakly unitary} if $\psi(v_i, v_j) = 0$ for all $i, j \in \{ 1, \dotsc, m \}$ such that $i \neq j$.
 If $\psi$ is non-degenerate, then this implies that $\psi(v_i, v_i) \neq 0$ for all~$i$.

We motivate the definition of signature by comparing it with Hermitian forms over $\mathbb{C}$. Let us say that $ H(x, y) $ is a Hermitian form on a finite-dimensional complex vector space. You can associate a matrix $ A $ to $ H $, such that:
\[
H(x, y) = x^* A y
\]
where $ x^* $ is the conjugate transpose of $ x $, and $ A $ is a Hermitian matrix (i.e., $ A = A^* $).

\noindent We know that any Hermitian matrix can be diagonalized with real eigenvalues. So, after a suitable basis change, the Hermitian form becomes
\[
H(x, x) = \sum_{i=1}^p |x_i|^2 - \sum_{j=1}^q |x_j|^2
\]
for some $ p, q \geq 0 $ and $ p + q = \dim V $.

\begin{definition}
    The \emph{signature} of the Hermitian form $ H $ is the pair $ (p, q) $, where $p,q$ are the dimensions of maximal positive and
negative definite subspaces, respectively. For simplicity, we will only consider $p$ for signature as $q$ is determined by $p$ for non-degenerate forms.
\end{definition} 

\subsection*{Example}

Let $ H : \mathbb{C}^2 \times \mathbb{C}^2 \rightarrow \mathbb{C} $ be defined by:
\[
H(x, y) = x_1 \overline{y_1} - x_2 \overline{y_2}
\]
Then:
\[
H(x, x) = |x_1|^2 - |x_2|^2
\]
So the signature of $ H $ is $ (1, 1) $.

\begin{definition}
Let $V_i$ be a left $\rM_d(\mathbb{C})$-vector space equipped with a non-degenerate $\rM_d(\mathbb{C})$-skew-Hermitian form $\psi_i \colon V_i \times V_i \to \rM_d(\mathbb{C})$.
Let $v_1, \dotsc, v_m$ be a left $\rM_d(\mathbb{C})$-basis for $V$ which is weakly unitary. Let $\alpha_{j1}, \cdots,  \alpha_{jd}$ be the eigenvalues of $\psi_i(v_j,v_j)$ repeated with their algebraic multiplicities.
    We say $\Psi_i$ has \textbf{signature} $r_i$, where $r_i$ is the number of $\alpha_{jl}(1\leqslant j \leqslant m, 1\leqslant l \leqslant d)$ s.t $\textnormal{Im}(\alpha_{jl})>0$.
\end{definition}

\begin{definition}
   Let $(D,\dag)$ be a type IV division $Q$-algebra with positive involution.
Let $V$ be a left $D$-vector space equipped with a non-degenerate $(D,\dag)$-skew-Hermitian form $\psi \colon V \times V \to D$.
Let $v_1, \dotsc, v_m$ be a left $D$-basis for $V$ which is weakly unitary.Let $\alpha^{-1} \colon (\rM_d((\mathbb{C}))^{e},t) \to (D_\bR,\dag)$ be an isomorphism of $\bR$-algebras with involution.
Let $\psi \otimes \mathbb{R} = (\psi_1, \cdots, \psi_e)$ and $V= \oplus_{i=1}^{e}V_i$, where $\psi_i$ is a $\rM_d(\mathbb{C})$-skew-Hermitian form on $V_i$ for $1 \leqslant i \leqslant e$. We say $\psi$ has $\alpha$-\textbf{signature} $(r_1, \cdots r_e)$ where $\psi_i$  has \textbf{signature} $r_i$.
\end{definition}

\begin{definition}
    We define a \textbf{signature} to be string of $e$-integers $(r_1, \cdots ,r_e)$, where $0 \leq r_i \leq dm$ for all $1 \leq i \leq e$. 
\end{definition}

\begin{definition}
    Let $D_0 = \rM_d(\mathbb{C})^e$. Let $\psi_0$ be a nondegenerate $(D_0,t)$-skew-Hermitian form. A \textbf{sign matrix} for $\psi_0$ is a block matrix that belongs to $\rM_{(edm)*(dm)}(\mathbb{C})$ of the form,
    \[ (\varepsilon)=\begin{pmatrix}
         \varepsilon_{11} & 0 & 0 & \cdots  0 \\ 
           0 & \varepsilon_{12} & 0 & \cdots  0 \\ 
          0 & 0 & \varepsilon_{13} & \cdots  0  \\ 
         0 & 0 & 0 & \cdots  \varepsilon_{1m}  \\ 
         \varepsilon_{21} & 0 & 0 & \cdots  0 \\ 
           0 & \varepsilon_{22} & 0 & \cdots  0 \\ 
          0 & 0 & \varepsilon_{23} & \cdots  0  \\ 
         0 & 0 & 0 & \cdots  \varepsilon_{2m}  \\
          \vdots & \vdots & \vdots & \ddots \\
          \varepsilon_{e1} & 0 & 0 & \cdots  0 \\ 
           0 & \varepsilon_{e2} & 0 & \cdots  0 \\ 
          0 & 0 & \varepsilon_{e3} & \cdots  0  \\ 
         0 & 0 & 0 & \cdots  \varepsilon_{em}  \\ 
    \end{pmatrix}, \]  where $\varepsilon_{ij}$ is a $d \times d$ diagonal matrix whose entries are $\pm i$ such that for each fixed $i$, the total number of diagonal entries of $\varepsilon_{ij}=+i$ when $j$ runs from $1 \textnormal{ to }  m$ occurs $r_i$ times, where $r_i$ is a signature. 
\end{definition}

For a general division algebra with involution~$(D,\dag)$, there is no canonical choice of a non-zero element of $D^-$, so there is no natural definition of ``unitary basis'' with respect to a $(D,\dag)$-skew-Hermitian form.

\begin{definition}
In the case $D_0 = \mathbb{H}^e$,  If $V$ is a free left $D_0$-module equipped with a $(D_0,t)$-skew-Hermitian form $\psi_0$. Let $\alpha^{-1} \colon (\mathbb{H}^e,t) \to (D_\bR,\dag)$ be an isomorphism of $\bR$-algebras with involution, then we say that a left $D_0$-basis $v_1, \dotsc, v_m$ of $V$ is $\alpha(\varepsilon)-$\textbf{unitary} if it is weakly unitary and $\alpha(\psi(v_i, v_i)) = (i, \cdots , i)$.

In the case $D_0 = \rM_d(\mathbb{C})^e$,  If $V$ is a free left $D_0$-module equipped with a $(D_0,t)$-skew-Hermitian form $\psi_0$. Let $\alpha^{-1} \colon (\rM_d((\mathbb{C}))^{e},t) \to (D_\bR,\dag)$ be an isomorphism of $\bR$-algebras with involution, then we say that a left $D_0$-basis $v_1, \dotsc, v_m$ of $V$ is $\alpha(\varepsilon)-$\textbf{unitary} for a given sign matrix $(\varepsilon)$ if it is weakly unitary and $\alpha(\psi(v_i, v_i)) = (\varepsilon_{1i}, \cdots , \varepsilon_{ei})$.
\end{definition} 

\noindent \textbf{Note}: It is important to consider the $\alpha(\varepsilon)$-\textbf{unitary} basis, as \cref{conjclass1} does not hold without fixing the signature. In the type III case, the signature does not play a role, so it would be sufficient to refer to the basis as simply $\alpha$-\textbf{unitary}. However, for the sake of uniformity with the type IV case, we adopt the notation $\alpha(\varepsilon)$-\textbf{unitary} throughout.  Our definition in the type IV case is necessarily more intricate in order to fix the signature, as different bases can yield different sign
matrices for the same skew-Hermitian form. Without fixing the signature at this stage, subsequent calculations become significantly more complex.

    

\begin{remark}\label{WRem}
    The definition of signature of a $\rM_d(\mathbb{C})$-skew-Hermitian form involves choosing a weakly unitary basis.  In order for this to be well-defined, we need to know two things: there always exists a weakly unitary basis, and the signature is independent of the choice of weakly unitary basis. 
\end{remark}

\begin{lemma}\label{signwd}
Let $D_0 = \rM_d(\mathbb{C})^e$ and let $t$ denote the involution of $D_0$ which is transpose and complex conjugate on each factor.
Let $V$ be a free left $D_0$-module and let $\psi_0$ be a non-degenerate $(D_0,t)$-skew-Hermitian form $V \times V \to D_0$. Let $\alpha^{-1} \colon ((\rM_d(\mathbb{C}))^{e},t) \to (D_\bR,\dag)$ be an isomorphism of $\bR$-algebras with involution.
\begin{enumerate}
    \item Then there exists a $D_0$-basis $v_1, \dotsc, v_m$ for $V$ which $\{ v_1, \dotsc, v_m \}$ is weakly unitary w.r.t $\psi_0$. 
    \item Given two weakly unitary bases $\{ e_1, \dotsc, e_m \} \text{ and } \{ f_1, \dotsc, f_m \}$ of $V$ w.r.t. $\psi_0$. Let  $[\psi]_{e}$ and $[\psi]_{f} \in \rM_{dm}(\mathbb{C})$  denote the matrix representations made up of blocks $\psi(e_i,e_j)$, $\psi(f_i,f_j)$ in the bases $\{e_i\}$ and $\{f_i\}$, respectively. Then the number of positive imaginary eigenvalues of $[\psi]_{e}$ and $[\psi]_{f}$ is the same.
\end{enumerate}
\end{lemma}

\begin{proof}
 Write $B_0 = \rM_d(\mathbb{C})$. Let $u_1, \dotsc, u_{e}$ denote the standard $\mathbb{R}$-basis of $D_0 = \rM_d(\mathbb{C})^e$.

Let $V_i = u_i V$.
Then $V = \bigoplus_{i=1}^{e} V_i$ and each $V_i$ is a free left $B_0$-module.
Because $V$ is a free left $D_0$-module, $\textnormal{rk}_{B_0}(V_1) = \dotsb = \textnormal{rk}_{B_0}(V_{e})$. Let $m$ denote this rank.  Because $\psi : V \times V \to D_0$ is $\mathbb{C}$-bilinear, it takes the form \[ \psi((x_1, \dotsc, x_{e}), (y_1, \dotsc, y_{e})) = (\psi_1(x_1, y_1), \dotsc, \psi_e(x_{e}, y_{e})) \text{ for all } x_i, y_i \in V_i, \] where $\psi_i : V_i \times V_i \to B_0$ are some non-degenerate $(B_0, t)$-skew-Hermitian forms.

We prove by induction on $m$ that there is a weakly-unitary $B_0$-basis for $V_i$, $v_{i1}, \dotsc, v_{im}$ using the Gram-Schmidt method.
  First we \textbf{claim} that there exists $z \in V_i$ such that $\psi_i(z,z) \in \textnormal{Gl}_d(\mathbb{C})$. We will prove this claim after the proof of this lemma.
Since skew hermitian form are diagonalisable, so there exists $D \in Gl_d(\mathbb{C})$ s.t 
\[\psi(Dz,Dz)= \diag(\pm\lambda_1i, \cdots \pm\lambda_di), \]
with $\lambda_1, \cdots ,\lambda_d >0$.
Take 
\[B = \diag\bigg(\frac{1}{\sqrt\lambda_1}, \cdots \frac{1}{\sqrt\lambda_d} \bigg),\]
then $\psi(BDz,BDz)=\diag(\pm i, \cdots \pm i)$. Taking $J=BD$, there exists $J$ s.t $\psi_i(Jz,Jz) = \tilde{I}_s$, where $\tilde{I}_s$ is a invertible $d$ by $d$ matrix with diagonal entries $\pm i$.
  Let $v_{i1}=Jz$. 
 
Let $V_i' = \{ v \in V_i : \psi_i(v_{i1}, v) = 0 \} = \{ v \in V_i : \psi_i(v, v_{i1}) = 0 \}$, which is a left $B_0$-submodule of $V_i$.
 For every $b \in B_0 \setminus \{0\}$, we have
\begin{equation} \label{eqnb1}
\psi_i(bv_{i1}, v_{i1}) = b\psi_i(v_{i1}, v_{i1}) \neq 0
 \end{equation}
and so $B_0v_{i1} \cap V_i' = \{ 0 \}$.
 For every $v \in V_i$, we have
\[ v - \psi_i(v, v_{i1}) \psi_i(v_{i1}, v_{i1})^{-1} v_{i1} \in V_i'. \]
Hence $V_i = B_0v_{i1} \oplus V_i'$ as a direct sum of left $B_0$-modules.

By \eqref{eqnb1}, $bv_{i1} \neq 0$ for all $b \in B_0 \setminus \{0\}$. Hence $\dim_\mathbb{R}(B_0v_{i1}) = 2d^2$ and so $\dim_k(V_i') = 2d^2(m-1)$.
Every $B_0$-module whose $\mathbb{R}$- dimension is a multiple of $2d^2$ is a free $B_0$-module, so $B_0v_{i1}$ and $V_0'$ are free left $B_0= \rM_d(\mathbb{C})$-modules.
 By induction, there is a weakly unitary $B_0$-basis $v_{i2}, \dotsc, v_{im}$ for $V_i'$. This completes the proof of part 1.

For a weakly unitary basis $\{e_i\}$ of $V$, the signature of the skew-Hermitian form $\psi$ is defined by counting the number of positive, negative, and zero eigenvalues of the associated matrix representation of $\psi$ in this basis (nonzero only when $\psi$ is non-degenerate). Consider two weakly unitary bases $\{e_i\}$ and $\{f_i\}$ of $V$. Let $P \in \mathrm{M}_m(\mathrm{M}_d(\mathbb{C})) \cong \mathrm{M}_{md}(\mathbb{C})$ be the transition matrix relating these bases:
\[
e_i = \sum_j P_{ij} f_j,
\]
where $P_{ij} \in \mathrm{M}_{d}(\mathbb{C})$. The matrix representation of $\psi$ with respect to the basis $\{e_i\}$ is related to its representation with respect to the basis $\{f_i\}$ via the transition matrix:
\[
[\psi]_e = P^* [\psi]_f P,
\]
where $[\psi]_e$ and $[\psi]_f$ denote the matrices made up of blocks of $\psi(e_i, e_j)$ and $\psi(f_i, f_j)$ in the respective bases. The signs of the imaginary parts of the eigenvalues of such matrices are invariant under this transformation, ensuring that the signature is preserved.

\end{proof}

\begin{claim}
  Let $\psi_i$ and $V_i$ be as in the proof of \cref{signwd}. Then there exists $z \in V_i$ such that $\psi_i(z,z) \in \textnormal{Gl}_d(\mathbb{C})$. 
\end{claim}
\begin{proof}\label{QQ}  By \cref{2ideal}, we can choose $x,y \in V_i$ such that $\psi_i(x,y) = \textnormal{I}_d$ as $\psi  $ is non-degenerate (There is a minor gap at the corresponding step in the proof of \citep[Lemma 3.5, Case $d=2$, part~(i)]{Lattices}.  It is claimed there that the image of
$\psi_i : V_i \times V_i \to B_0$ is a two-sided ideal in $B_0$, and
this is used to prove that the image is all of $B_0$.  However, it is
not easy to prove that the image of $\psi_i$ is a two-sided ideal until \emph{after} proving that it is all of $B_0$ (it is not obvious that
$\psi_i(V_i \times V_i)$ is closed under addition). This gap in \citep{Lattices} can be fixed by applying this paper's \cref{2ideal} to show directly that $\psi_i(V_i \times V_i) = B_0$.). 

Now there exists $Q$ s.t $\psi_i(Qy,Qy)= \textnormal{R}$, where R is a diagonal matrix with entries $+i$, followed by $-i$ together making up the first $r$ diagonal entries and remaining $d-r$ entries $0$. We get  $\psi_i((Q^{\dag})^{-1}x,Qy) = \textnormal{I}_d$. Now let  $\psi_i((Q^{\dag})^{-1}x,(Q^{\dag})^{-1}x)= \textnormal{P}$. Letting $X=(Q^{\dag})^{-1}x$ and $Y=Qy$. If either $P$ or $R$ is invertible then we are done. If not then we know we can write $R$ and $P$ in blocks of size $r, d-r$ as follows,
\[R = \begin{pmatrix}
\tilde{I}_r & 0 \\
0 & 0 
\end{pmatrix} , P = \begin{pmatrix}
P_1 & P_2 \\
-{P_2}^{\dag} & C^{\dag}-C 
\end{pmatrix},\]
so let $z=X+UY$, where 
\[U= \begin{pmatrix}
K & P_2 \\
0 & D 
\end{pmatrix} \] s.t $D=-C-\frac{i}{2} \textnormal{I}_{d-r}$ and $P_1= \begin{pmatrix}
    A_1 & A_2 \\
    -A_2^{\dag} & A_3
\end{pmatrix}$, where $A_1$ has same size as number of positive $i$ in $\tilde{I}_r$ and similarly $K= \begin{pmatrix}
    K_1 & K_2 \\
    0 & K_3
\end{pmatrix}$, where $K_1,K_2$ and $K_3$ will be chosen later.
Note that $(C+D)^{\dag}-(C+D)=i\textnormal{I}_{d-r}$. We have
 \begin{align}\label{I}
     \psi_i(X+UY,X+UY) &= \psi_i(X,X) + \psi_i(X,Y)U^{\dag} + U\psi_i(Y,X) + U\psi_i(Y,Y) U^{\dag} \\ \nonumber
     &= P + U^{\dag} - U +URU^{\dag} \\ \nonumber
     &= \begin{pmatrix}
P_1+K^{\dag}-K+K\tilde{I}_rK^{\dag} & P_2-B \\
-(P_2-B)^{\dag} & (C+D)^{\dag}-(C+D) 
\end{pmatrix} \\
&= \begin{pmatrix}
    L & 0 \\
    0 & i\textnormal{I}_{d-r}
\end{pmatrix},
 \end{align}
 where $L =  \begin{pmatrix}
     A_1-iK_2K_2^{\dag}+K_1^{\dag}-K_1+iK_1K_1^{\dag} & A_2-K_2(iK_3^{\dag}+I) \\
     -(A_2-K_2(iK_3^{\dag}+I))^{\dag} & A_3 +K_3^{\dag}-K_3-iK_3K_3^{\dag}
 \end{pmatrix}$.

Choose $K_3=\frac{1}{2} A_3+i  tM $, where $M$ is positive definite hermitian matrix and $t$ is really large real number such that $K_3$ is positive definite  skew hermitian (meaning $iK_3^{\dag}$ is positive definite hermitian) as in \cref{PSP}. Now $(iK_3^{\dag}+I)$ is invertible as $iK_3^{\dag}$ is positive definite hermitian and so is $I$. As $A_3 =\frac{1}{2} (A_3-{A_3}^{\dag})$, we have $ A_3 +K_3^{\dag}-K_3-iK_3K_3^{\dag}= -i(2 t M +K_3K_3^{\dag})$ which is invertible, since $K_3K_3^{\dag}$ is semi-positive definite hermitian and $2tM$ is positive definite hermitian. Now choose $K_2=A_2 (iK_3^{\dag}+I)^{-1}$. Now taking $A_1+iK_2K_2^{\dag}=C_1-C_1^{\dag}$ for some $C_1$, and choose $K_1=C_1-i t M$, we have $ A_1+iK_2K_2^{\dag}+K_1^{\dag}-K_1+iK_1K_1^{\dag}= i(2  tM +K_1K_1^{\dag})$ which is invertible by same argument as earlier. 
\linebreak \noindent Hence $L= \begin{pmatrix}
 i(2 t M +K_1K_1^{\dag} & 0\\
 0& -i(2 t M +K_3K_3^{\dag}))
\end{pmatrix}$ is invertible.  
Choosing $z = X+UY$, we obtain $\psi_i(z,z) \in \textnormal{Gl}_d(\mathbb{C})$.
\end{proof}

\begin{lemma}{\citep[Lemma 3.1]{Lattices}} \label{tr-skew-hermitian-form}
Let $(D, \dag)$ be a semisimple $k$-algebra with an involution.
Let $V$ be a left $D$-module.
Then the map $\psi \mapsto {\Trd_{D/k}} \circ \psi$ is a bijection between the set of $(D,\dag)$-skew-Hermitian forms on $V$ and the set of $(D,\dag)$-compatible skew-symmetric forms on~$V$.
\end{lemma}

\begin{remark}
Note that for a fixed $d,e,m$, there are finitely many \textbf{sign matrix} associated to it.   
\end{remark}

\noindent Example: Consider $\psi(v_1,v_2)=v_1A{v_2}^{\dag}$ a $\rM_2(\mathbb{C})$-skew-Hermitian form (d=2,m=2,e=1), where $A=[\psi]= \begin{pmatrix}
    \begin{pmatrix}
         i & 0 \\ 0& -i
    \end{pmatrix} & \begin{pmatrix}
         0 & 0 \\ 0& 0      
    \end{pmatrix} \\
    \begin{pmatrix}
         0 & 0 \\ 0& 0      
    \end{pmatrix} & \begin{pmatrix}
         i & 0 \\ 0& -i
    \end{pmatrix}
    \end{pmatrix}$. 
Consider standard basis $e_1= \begin{pmatrix}
    1 & 0 & 0 & 0 \\
    0 & 1 & 0 & 0
\end{pmatrix},e_2=\begin{pmatrix}
   0 & 0 & 1 & 0 \\
    0 & 0 & 0 & 1
\end{pmatrix}$. Then $\psi(e_1,e_1)=\begin{pmatrix}
    i & 0 \\ 0& -i
\end{pmatrix}$ and $\psi(e_2,e_2)=\begin{pmatrix}
    i & 0 \\ 0& -i
    \end{pmatrix}$. So we  get signature 2.
    Consider another basis $f_1=\begin{pmatrix}
    1 & 0 & 0 & 0 \\
    0 & 0 & 1 & 0
\end{pmatrix},f_2=\begin{pmatrix}
   0 & 1 & 0 & 0 \\
    0 & 0 & 0 & 1
\end{pmatrix}$. Then $\psi(f_1,f_1)=\begin{pmatrix}
    i & 0 \\ 0& i
\end{pmatrix}$ and $\psi(f_2,f_2)=\begin{pmatrix}
    -i & 0 \\ 0& -i
\end{pmatrix}$. So we get signature 2. 
We get unitary matrix $P= 
\begin{pmatrix}
1 & 0 & 0 & 0 \\
0 & 0 & 1 & 0 \\
0 & 1 & 0 & 0 \\
0 & 0 & 0 & 1
\end{pmatrix}$, such that $[\psi]_{e_1,e_2} = P^{\dag} [\psi]_{f_1,f_2} P$ that matches our previous proof.

The following lemma is an adaptation for types III and IV from \citep[Lemma 3.4]{Lattices}, which was originally stated for types I and II and used in the proof of \citep[Proposition 7.1]{Lattices}. With this adaptation and other generalisation, generalised \citep[Proposition 7.1]{Lattices} holds in our type III and IV case as well. 
\begin{lemma} \label{semi-orthogonal-normalise}
Let $(D,\dag)$ be a division $\mathbb{Q}$-algebra with positive involution of type III or IV.
Let $\alpha^{-1} \colon ((\mathbb{H})^{e},t) \to (D_\bR,\dag)$ be an isomorphism of $\bR$-algebras with involution for type III and let $\alpha^{-1} \colon ((M_d(\mathbb{C}))^{e},t) \to (D_\bR,\dag)$ be an isomorphism of $\bR$-algebras with involution for type IV. 
Let $V$ be a left $D$-vector space equipped with a non-degenerate $(D,\dag)$-skew-Hermitian form $\psi \colon V \times V \to D$.
Let $v_1, \dotsc, v_m$ be a left $D$-basis for $V$ which is weakly unitary.

Then there exist $s_1, \dotsc, s_m \in D_\bR^\times$ and  a sign matrix $(\varepsilon)$ (for type IV)
such that $s_1^{-1} v_1, \dotsc, s_m^{-1} v_m$ form a $\alpha(\varepsilon)$-unitary $D_\bR$-basis for $V_\bR$ and, for all $i$,
\[ \abs{s_i}_D \leq (2ke)^{1/4} \abs{\psi(v_i, v_i)}_D^{1/2} \],
where $k=1$ for type III and $k=d$ for type IV
\end{lemma}

\begin{proof}
Type III) For each~$i$, we have $\psi(v_i, v_i) \in D^- \setminus \{0\}$. 
Thus, we may write $\psi(v_i, v_i) = y_i$ for some $y_i$ such that
\[
\alpha(y_i) = (y_{i1}, \dotsc, y_{ie}) \in (i\mathbb{R} + j\mathbb{R} + k\mathbb{R})^e \setminus \{\mathbf{0}\},
\]
where each component $y_{ij}$ has the form
\[
y_{ij} = a i + b j + c k,
\]
for some real numbers $a, b, c \in \mathbb{R}$ (depending on $i$ and $j$, though we omit writing $a_{ij}, b_{ij}, c_{ij}$ for ease of notation), not all zero. Now let \[x_{ij}= (\sqrt{a^2+b^2+c^2}+a+cj-bk) \textnormal{ and } \alpha_{ij}= \norm{x_{ij}}^2 \sqrt{a^2+b^2+c^2}, \] 
then routine calculation shows that $x_{ij}y_{ij}x_{ij}^{\dag}=\alpha_{ij} \cdot i$.
Let $s_i = \alpha^{-1}(s_{i1}, \dotsc, s_{ie}) \in D_\bR^\times$, where $s_{ij}= \frac{\sqrt{{\alpha}_{ij}}\cdot x_{ij}^{\dag}}{\norm{x}_{ij}^2}$.

Hence,
\[ \psi(s_{ij}^{-1} v_i, s_i^{-1} v_i) = s_{ij}^{-1} \psi(v_i, v_i) (s_{ij}^{-1})^\dag = (\frac{x_{ij}}{\sqrt{{\alpha}_{ij}}}) \cdot y_{ij} \cdot (\frac{x_{ij}}{\sqrt{{\alpha}_{ij}}})^{\dag}  = \alpha_{ij}/\alpha_{ij} \cdot i=i. \]

Furthermore,
\begin{align*}
    \abs{s_i}_D^2
  & = \sum_{j=1}^e \Trd_{\mathbb{H}/\bR}(s_{ij} s_{ij}^t)
    = \sum_{j=1}^e  \Trd_{\mathbb{H}/\bR}\Bigg(\frac{\sqrt{{\alpha}_{ij}}\cdot x_j^{\dag}}{\norm{x_{ij}}^2} \cdot \frac{\sqrt{{\alpha}_{ij}}\cdot x_{ij}}{\norm{x_{ij}}^2} \Bigg)
\\& =\sum_{j=1}^e  \Trd_{\mathbb{H}/\bR}\Bigg( \frac{\alpha_{ij}}{\norm{x_{ij}}^2} \Bigg)=\sum_{j=1}^e 2 \sqrt{y_{ij}y_{ij}^{\dag}}
\\& \leq \sqrt{2e \sum_{j=1}^e \abs{y_{ij}}^2}
    = \sqrt{2e \Trd_{D_\bR/\bR}(y_{i}y_{i}^\dag)}
    = (2e)^{1/2} \abs{y_i}_D.
\end{align*}
As $\abs{a \alpha(\omega_i)}_{D_0}= \abs{a}_{D_0}$ , this implies that
\[\abs{s_i}_D^2 \leq (2e)^{1/2} \abs{y_i\alpha(\omega_i)}_D = (2e)^{1/2} \abs{\psi(v_i, v_i)}_D. \] 

For type~IV, for each index~$i$, define $A_i = \psi(v_i, v_i) \in D^- \setminus \{0\}.$
Our goal is to find an element $ s_i \in D_{\mathbb{R}}^* $ such that $s_i^{-1} A_i (s_i^{-1})^{\dag} = \alpha(\omega_i),$
for a suitable element $ \alpha(\omega_i) \in D^- $, diagonal with entries encoding the signature. To construct such an element $ s_i $, proceed as follows. Write
\[
A_i = (A_{i1}, \ldots, A_{ie}),
\]
where each $ A_{ij} \in \mathrm{M}_d(\mathbb{C}) $ is skew-Hermitian. Since skew-Hermitian matrices are diagonalisable via unitary conjugation, we may express each as
\[
A_{ij} = P_{ij} D_{ij} P_{ij}^{\dag},
\]
where $ D_{ij} $ is diagonal with purely imaginary entries, and $ P_{ij} \in \mathrm{U}_d(\mathbb{C}) $, so that $ P_{ij}^{-1} = P_{ij}^{\dag} $. For each $ 1 \leq l \leq d $, define
\[
t_{ijl} := \sqrt{|D_{ijl}|} \in \mathbb{R}, \qquad \varepsilon_{ijl} := \operatorname{sign}(D_{ijl}) \cdot i \in \{\pm i\}.
\]
Let $ t_{ij} $ be the diagonal matrix with diagonal entries $ t_{ijl} $ (Since $t_{ijl} \cdot \varepsilon_{ijl} \cdot \overline{t_{ijl}} = D_{ijl}$). We define $
s_{ij} := P_{ij} t_{ij}.$
We may then write
\[
A_{ij} = s_{ij} \varepsilon_{ij} s_{ij}^{\dag},
\]
where $ \varepsilon_{ij} \in \mathrm{M}_d(\mathbb{C}) $ is diagonal with entries $ \pm i $, encoding the signature.
From the diagonalisation,
\[
D_{ij} = P_{ij}^{\dag} s_{ij} \varepsilon_{ij} s_{ij}^t P_{ij} = t_{ij} \varepsilon_{ij} t_{ij}^t.
\] 
Finally, define $
s_i := \alpha^{-1}(s_{i1}, \ldots, s_{ie}) \in D_{\mathbb{R}}^*.$
This construction guarantees that
\[
s_i^{-1} A_i (s_i^{-1})^{\dag} = \alpha(\omega_i).
\]
Moreover, we observe that $
s_i s_i^{\dag} = P_i t_i t_i^{\dag} P_i^{\dag} = P_i t_i t_i^{\dag} P_i^{-1},$
so taking traces yields
$
\operatorname{Tr}(s_i s_i^{\dag}) = \operatorname{Tr}(t_i t_i^{\dag}).
$

Hence,
\begin{align*}
    \abs{s_i}_D^2
  & = \sum_{j=1}^e \Trd_{D_\mathbb{R}/\mathbb{R}}(s_{ij} {s_{ij}}^{t})
  =\Trd_{D_\mathbb{R}/\mathbb{R}}(t_{ij} t_{ij}^{t})
  = \sum_{j=1}^e \sum_{l=1}^d \Trd_{\mathbb{C}/ \mathbb{R}}(\abs{t_{ijl}}^2)
 \\& =\sum_{j=1}^e \sum_{l=1}^d 2t_{ijl}^2
     \leq \sqrt{4de \sum_{j=1}^e \sum_{l=1}^d t_{ijl}^4} =(2de)^{1/4} \sqrt{\Trd_{D_\mathbb{R}/\mathbb{R}}(A_{i} {A_{i}}^t})
 \\&    = (2de)^{1/4} \abs{\psi(v_i, v_i)}_D.
\end{align*}

\end{proof}

The following two lemmas are stated for types III and IV, adapted from \citep[Lemma 3.5]{Lattices}, which was originally formulated for types I and II and used in the proof of \citep[Lemma 5.1]{Lattices}. We will utilize these lemmas to establish \cref{conjclass1}.
\begin{lemma} \label{D0-basis1} \leavevmode
Let $D_0 = \mathbb{H}^e$ and let $t$ denote the involution of $D_0$ which is $\Trd(a)-a$ on each factor.
Let $V$ be a free left $D_0$-module and let $\psi_0$ be a non-degenerate $(D_0,t)$-skew-Hermitian form $V \times V \to D_0$. Let $\alpha^{-1} \colon ((\mathbb{H})^{e},t) \to (D_\bR,\dag)$ be an isomorphism of $\bR$-algebras with involution. 

Then there exists a $D_0$-basis $v_1, \dotsc, v_m$ for $V$ and a $\mathbb{R}$-basis $a_1, \dotsc, a_{4e}$ for $D_0$ with the following properties:
\begin{enumerate}[(i)]
\item $\{ v_1, \dotsc, v_m \}$ is $\alpha$-unitary with respect to $\psi_0$.
\item $\{ a_1, \dotsc, a_{4e} \}$ is an orthonormal basis for $D_0$ with respect to $\abs{\cdot}_D$.
\item $\{ a_r v_j : 1 \leqslant r \leqslant 4e, 1 \leqslant j \leqslant m \}$ is a symplectic $\mathbb{R}$-basis for $V$ with respect to $\Trd_{D_0/\mathbb{R}} \psi$.
\end{enumerate}
\end{lemma}
\begin{proof}
    Write $B_0 = \mathbb{H}$. Let $u_1, \dotsc, u_{e}$ denote the standard $\mathbb{R}$-basis of $D_0 = \mathbb{H}^e$.

Let $V_i = u_i V$.
Then $V = \bigoplus_{i=1}^{e} V_i$ and each $V_i$ is a free left $B_0$-module.
Because $V$ is a free left $D_0$-module, $\textnormal{rk}_{B_0}(V_1) = \dotsb = \textnormal{rk}_{B_0}(V_{e})$. Let $m$ denote this rank.  Because $\psi : V \times V \to D_0$ is $\mathbb{H}$-bilinear, it takes the form \[ \psi((x_1, \dotsc, x_{e}), (y_1, \dotsc, y_{e})) = (\psi_1(x_1, y_1), \dotsc, \psi_e(x_{e}, y_{e})) \text{ for all } x_i, y_i \in V_i, \] where $\psi_i : V_i \times V_i \to B_0$ are some non-degenerate $(B_0, t)$-skew-Hermitian forms.

We prove by induction on $m$ that there is a unitary $B_0$-basis for $V_i$, $v_{i1}, \dotsc, v_{im}$ using the Gram-Schmidt method.
 First we \textbf{claim} that there exists $z \in V_i$ such that $\psi_i(z,z) \neq 0$. By \cref{2ideal}, we can choose $x,y \in V_i$ such that $\psi_i(x,y) =i$,
 \begin{align}\label{Ii}
     \psi_i(x+y,x+y) &= \psi_i(x,x) + \psi_i(x,y) + \psi_i(y,x) + \psi_i(y,y) \\ \nonumber
     &= \psi_i(x,x) + \psi_i(y,y) + \psi_i(x,y) - \psi_i(x,y)^{\dag} 
 \end{align}
Then $\psi_i(x,x)$, $\psi_i(y,y)$ and $\psi_i(x+y,x+y)$ are not all zero.
Choosing $z$ to be one of $x$, $y$ or $x+y$, we obtain $\psi_i(z,z) \neq 0$.
Then $\psi_i(z,z) \in B_0^- = (ai+bj+ck) \backslash \{0\} $, where $a,b,c \in \mathbb{R}$. So let $\psi_i(z,z) = y$ for some $y=ai+bj+ck$. Letting $v_{i1} = kz$, where $k=\frac{x}{\norm{x}\cdot \abs{y}^{1/4}}$ and $x=\sqrt{\abs{y}}+a+cj-bk$ we obtain that $\psi_i(v_{i1}, v_{i1}) =  i$ which proves our claim.

Then there exists $J$ s.t $\psi_i(Jz,Jz) = i$ . Let $v_{i1}=Jz$.

Let $V_i' = \{ v \in V_i : \psi_i(v_{i1}, v) = 0 \} = \{ v \in V_i : \psi_i(v, v_{i1}) = 0 \}$, which is a left $B_0$-submodule of $V_i$.
 For every $b \in B_0 \setminus \{0\}$, we have
\begin{equation} \label{eqnb11}
\psi_i(bv_{i1}, v_{i1}) = b\psi_i(v_{i1}, v_{i1}) \neq 0
 \end{equation}
and so $B_0v_{i1} \cap V_i' = \{ 0 \}$.
 For every $v \in V_i$, we have
\[ v - \psi_i(v, v_{i1}) \psi_i(v_{i1}, v_{i1})^{-1} v_{i1} \in V_i'. \]
Hence $V_i = B_0v_{i1} \oplus V_i'$ as a direct sum of left $B_0$-modules.

By \eqref{eqnb1}, $bv_{i1} \neq 0$ for all $b \in B_0 \setminus \{0\}$. Hence $\dim_\mathbb{R}(B_0v_{i1}) = 4$ and so $\dim_k(V_i') = 4(m-1)$.
Every $B_0$-module whose $\mathbb{R}$- dimension is a multiple of $4$ is a free $B_0$-module, so $B_0v_{i1}$ and $V_0'$ are free left $B_0= \mathbb{H}$-modules.
 By induction, there is a unitary $B_0$-basis $v_{i2}, \dotsc, v_{im}$ for $V_i'$.
 For part (ii) and (iii). Let $b_{1}= 1/\sqrt{2}$, $b_{2}= i/\sqrt{2}$, $b_{3}=k/\sqrt{2}$ and $b_{4}=j/\sqrt{2}$
 
Since $\psi_i$ is $(B_0,t)$-skew-Hermitian,
\[ \psi_i(b_rv_{ij}, b_{r'}v_{ij'}) = b_r\psi_i(v_{ij}, v_{ij'})b_{r'}^\dag  \]
Thus if $j \neq j'$, we obtain $\psi_i(b_rv_{ij}, b_{r'}v_{ij'}) = 0$.
If $j = j'$, then we can calculate for $k=1,2$
\[ \Trd_{B_0/\mathbb{R}} \psi_i(b_rv_{ij}, b_{r'}v_{ij'}) = 
\Trd_{\mathbb{H}/\mathbb{R}} b_r\psi_i(v_{ij}, v_{ij}){b_{r'}}^{\dag}   
= \begin{cases}
 1 &\text{if } (r,r') = (2k-1,2k),
 \\ -1 &\text{if } (r,r') = (2k,2k-1),
 \\ 0 &\text{otherwise}.
 \end{cases}
 \]
 Thus the bases $v_{i1}, \dotsc, v_{im}$ and $b_1, b_2, \cdots ,b_{4}$ satisfy (iii) for $(B_0, V_i, \psi_i)$.

\end{proof}



\begin{lemma} \label{D0-basis} \leavevmode
Let $D_0 = \rM_d(\mathbb{C})^e$ and let $t$ denote the involution of $D_0$ which is a transpose and complex conjugate on each factor.
Let $V$ be a free left $D_0$-module and let $\psi_0$ be a non-degenerate $(D_0,t)$-skew-Hermitian form $V \times V \to D_0$. Let $\alpha^{-1} \colon ((\rM_d(\mathbb{C}))^{e},t) \to (D_\bR,\dag)$ be an isomorphism of $\bR$-algebras with involution. Let $(\varepsilon)$ be a sign matrix for $\psi_0$.

Then there exists a $D_0$-basis $v_1, \dotsc, v_m$ for $V$ and a $\mathbb{R}$-basis $a_1, \dotsc, a_{2d^2e}$ for $D_0$ with the following properties:
\begin{enumerate}[(i)]
\item $\{ v_1, \dotsc, v_m \}$ is $\alpha(\varepsilon)$-unitary with respect to $\psi_0$ with sign matrix $(\varepsilon)$.
\item $\{ a_1, \dotsc, a_{2d^2e} \}$ is an orthonormal basis for $D_0$ with respect to $\abs{\cdot}_D$.
\item $\{ a_r v_j : 1 \leqslant r \leqslant 2d^2e, 1 \leqslant j \leqslant m \}$ is a symplectic (upto reordering) $\mathbb{R}$-basis for $V$ with respect to $\Trd_{D_0/\mathbb{R}} \psi$.
\end{enumerate}
\end{lemma}

\begin{proof}
We repeat the first part of the proof of \cref{WRem} to get weakly unitary basis, infact unitary (as we are considering non-degenerate skew Hermitian form) and second part of the proof of \cref{WRem} to get the desired $\alpha(\varepsilon)-$unitary basis.
 For part (ii) and (iii). Let $b_{2k-1}$ be $d$ by $d$ matrix with $1/\sqrt{2}$ entry at $(k,k)$ and $0$ everywhere else and $b_{2k}= ib_{2k-1}$ for $ 1 \leqslant k \leqslant d^2$.
 
Since $\psi_i$ is $(B_0,t)$-skew-Hermitian,
\[ \psi_i(b_rv_{ij}, b_{r'}v_{ij'}) = b_r\psi_i(v_{ij}, v_{ij'})b_{r'}^\dag  \]
Thus if $j \neq j'$, we obtain $\psi_i(b_rv_{ij}, b_{r'}v_{ij'}) = 0$.
If $j = j'$, then we can calculate
\[ \Trd_{B_0/\mathbb{R}} \psi_i(b_rv_{ij}, b_{r'}v_{ij'}) = 
\Trd_{\rM_d(\mathbb{C})/\mathbb{R}} b_r\psi_i(v_{ij}, v_{ij}){b_{r'}}^{\dag}   
= \begin{cases}
 \pm 1 &\text{if } (r,r') = (2k-1,2k),
 \\ \mp 1 &\text{if } (r,r') = (2k,2k-1),
 \\ 0 &\text{otherwise}.
 \end{cases}
 \]
 Thus the bases $v_{i1}, \dotsc, v_{im}$ and $b_1, b_2, \cdots ,b_{2d^2e}$ satisfy (iii) for $(B_0, V_i, \psi_i)$.
\end{proof}

\subsection{Orders and discriminants}\label{orders}

Let $D$ be a simple $\mathbb{Q}$–algebra. By \citep[Theorem~10.1]{Rei75}, every element
$a \in R$ satisfies $\Trd_{D/\mathbb{Q}}(a) \in \mathbb{Z}$ for any order $R \subset D$.  
Since $\Tr_{D/\mathbb{Q}} = d\,\Trd_{D/\mathbb{Q}}$, this immediately implies that the
discriminant of $R$ is always an integer multiple of $d^{2d^2e}$ when $D$ is of
type~IV (the same argument applies for type~III, replacing $d^{2d^2e}$ with
$d^{d^2e}$). Thus
\begin{equation}\label{eqn:disc-lower-bound}
    \abs{\disc(R)} \ge d^{2d^2e}.
\end{equation}

Assume now that $(D,\dag)$ is a simple $\mathbb{Q}$–algebra endowed with a positive
involution.  By \citep[Lemma~5.6]{DO21b}, for any order $R\subset D$, the quantity
$\abs{\disc(R)}$ coincides with the discriminant of the symmetric $\mathbb{Q}$–bilinear
form
\[
(a,b)\ \longmapsto\ \Tr_{D/k}(ab^\dag).
\]
Therefore $\abs{\disc(R)}$ can be expressed as $d^{2d^2e}$ times the discriminant of
the positive definite form inducing $\abs{\cdot}_D$. Consequently,
\begin{equation}\label{eqn:disc-covol}
    \abs{\disc(R)} = d^{2d^2e}\,\covol(R)^2 .
\end{equation}

For any order $R\subset D$, its trace–dual is
\[
R^* = \{\, a\in D : \Trd_{D/\mathbb{Q}}(ab)\in\mathbb{Z}\ \text{for all } b\in R \,\}.
\]

The next few lemmas generalise those of \citep[Section~2.F]{Lattices} from the full
centre $F$ to an arbitrary subfield $F_0\subset F$, with only superficial changes.

\begin{lemma}\label{index-RcapF}
Let $D$ be a semisimple $k$–algebra, $R$ an order in $D$, and $F$ its centre.
Let $F_0\subset F$ be a subfield and $\mathcal{O}$ an order in $F_0$ containing
$R\cap F_0$. Then
\[
[\mathcal{O}:R\cap F_0]^2\ \abs{\disc(\mathcal{O}R)}  \abs{\disc(R)} .
\]
\end{lemma}

\begin{lemma}\label{dual-ideal}
Let $D$ be a simple $\mathbb{Q}$–algebra with centre $F$, and let $F_0\subset F$ be a
subfield with maximal order $\mathcal{O}_{F_0}$.  If $S$ is an order in $D$ containing
$\mathcal{O}_{F_0}$, and $S^*$ is defined as above, then there exists an ideal
$I\subset\mathcal{O}_{F_0}$ such that $IS^*\subset S$ and
\[
\Nm(I)\ \le\ d^{-2d^2e}\,\abs{\disc(S)} .
\]
\end{lemma}

\begin{lemma}\label{conductor-S}
Let $D$ be a simple $\mathbb{Q}$–algebra with centre $F$, and
$F_0\subset F$ as above with maximal order $\mathcal{O}_{F_0}$.
Let $R$ be an order in $D$, and set $S=\mathcal{O}_{F_0}R$.
Let $\fc$ be the conductor of the order $R\cap F_0$ in $\mathcal{O}_{F_0}$.
Then
\[
\fc S \subset R \qquad\text{and}\qquad \fc R^* \subset S^* .
\]
\end{lemma}

\begin{lemma}\label{L1}
Let $F$ be the centre of $D$, let $F_0$ be its totally real subfield, and let
$\mathcal{O}_{F_0}$ denote the ring of integers of $F_0$.
If $\fc$ is the conductor of $R\cap F_0$ inside $\mathcal{O}_{F_0}$, then
\[
\Nm(\fc) \le [\mathcal{O}_{F_0}:R\cap F_0]^2 .
\]
\end{lemma}
\pagebreak
\begin{lemma}\label{MI}
  Suppose that $L$ is a Dedekind domain and that $S$ is a free $L$-module. Then every submodule of $S$ is a direct sum of ideals of $L$.
\end{lemma}
\begin{proof}
See \citep[page 352]{AB74}
\end{proof}
\begin{lemma}\label{GCI}
     Let $K$ be a number field. Let $I$ be an ideal and $M$ be a submodule of $O_K$ of dim $n$ then
     \begin{align*}
         covol(IM)= \textnormal{Nm}(I)^n covol(M)
     \end{align*}
\end{lemma}
\begin{proof}
    Using \cref{MI} $M= \bigoplus_{i=1}^{n} M_i$ where $M_i$ are ideals of $O_K$, then $IM  \subseteq M$ and $M/IM \cong \bigoplus_{i=1}^{n} M_i/IM_i$. Hence $\abs{M/IM}= \textnormal{Nm}(I)^n$. Using this,
    \begin{align*}
        covol(IM)=[M:IM]covol(M)= \textnormal{Nm}(I)^ncovol(M)
    \end{align*}
\end{proof}
\subsection{Anti-symmetric elements in division algebras of types III and IV}\label{antisydmo}

For a division $\mathbb{Q}$--algebra $(D,\dag)$ with involution, set
\[
D^-=\{\, a\in D : a^\dag = -a \,\}.
\]
If $\psi$ is a $(D,\dag)$–skew-Hermitian form on a left $D$–module $V$,
then $\psi(x,x)\in D^-$ for all $x$, so the structure of $D^-$ plays an essential role
in the analysis of weakly unitary bases.

Assume now that $(D,\dag)$ is of Albert type~III with positive involution.
Under the real isomorphism $D_{\mathbb{R}}\cong\mathbb{H}^e$,
the subspace $D_{\mathbb{R}}^-$ corresponds to the purely imaginary elements of
$\mathbb{H}^e$, hence forms a free $F_{0,\mathbb{R}}$–module of rank $3$.
Thus $D^-$ is $3$–dimensional over $F_0$.

If $D$ is of type~IV, identify $D_{\mathbb{R}}\cong M_d(\mathbb{C})^e$.
Then $D_{\mathbb{R}}^-$ consists of the tuples whose real components are
skew-symmetric matrices with zero diagonal and whose complex components are symmetric.
This gives a free $F_{0,\mathbb{R}}$–module of rank $d^2$, so $D^-$ has $F_0$–dimension
$d^2$.

The next lemma is the analogue of \citep[Lemma~2.12]{Lattices} for division algebras of
type III and IV.

\begin{lemma}\label{as1}
Let $D$ be of type III or IV and let $\dag$ denote either $a^\dag=\Trd(a)-a$ or the
conjugate-transpose involution.
Let $R$ be an order in $D$, and fix an integer $\eta>0$ such that
$\eta R^\dag \subset R$.  
Then there exists $\omega\in D$ satisfying:
\begin{enumerate}[(i)]
    \item $\omega$ is nonzero and belongs to $D^-$;
    \item $\omega R^* \subset R$ and $R^*\omega\subset R$;
    \item $\abs{\omega}_D \le k\,\abs{\disc(R)}^{2l/e}$,
\end{enumerate}
where:
\[
k=
\begin{cases}
2^{-20/3}\,\gamma_{3e}^{1/2}\,\eta^{5/3}, & \text{type III}, \\
4\,\gamma_{d^2e}^{1/2}\,\eta^3\,d^{-4d^2}, & \text{type IV},
\end{cases}
\qquad
l=
\begin{cases}
3,& \text{type III},\\
1,& \text{type IV}.
\end{cases}
\]
\end{lemma}

\begin{proof}

Let $S = \mathcal{O}_{F_0} R$ and $S^- = S \cap D^-$.
Let $I$ be the ideal of $\mathcal{O}_{F_0}$ given by \cref{dual-ideal} applied to $S$.
Let $J = \fc^2 I$ (as a product of ideals of $\mathcal{O}_{F_0}$).
Then by \cref{conductor-S},
\begin{gather*}
   JSR^* = \fc SI \fc R^* \subset \fc SIS^* \subset \fc SS \subset \fc S \subset R,
\\ R^*JS = \fc I \fc R^* S \subset \fc IS^*S \subset \fc SS \subset \fc S \subset R.
\end{gather*}
Then we have $JS \cap D^- = JS^-$, as 
 $JS^- \subseteq JS$ and $J \subset D^+$
and for any $ j \in J$ and $s \in S^-$,
\[(js)^{\dag}= s^{\dag}j^{\dag}=-sj=-js,\]
so $JS^- \subseteq D^-$ and

  $JS^- \subseteq JS \cap D^-.$
Now for any $\alpha \in JS \cap D^-$, $\alpha=js$ and $\alpha^{\dag}=-\alpha$,
\[(js)^{\dag}=s^{\dag}j^{\dag}=js^{\dag}=-js,\]
so $s^{\dag}=-s$, hence $JS \cap D^- \subseteq JS^-$.

Hence if we choose $\omega \in JS \cap D^- \setminus \{0\} = JS^- \setminus \{ 0 \}$, then it will satisfy (i) and~(ii).

Since $S^-$ is a non-zero $\mathcal{O}_{F_0}$-submodule of an $F_0$-vector space of dimension $r =3$ for type III and $r=d^2$ for type IV, we use \cref{GCI} to conclude that
\[ \textnormal{covol}(JS^-) = \Nm(J)^{r} \textnormal{covol}(S^-), \]
where we measure covolumes in $D_\mathbb{R}^-$ by the volume form associated with the restriction of the inner product $\Trd_{D_\mathbb{R}/\mathbb{R}}(ab^\dag)$.

Let $S^+ = \{ a \in S : a^\dag = a \}$.
Then $S^+ \cap S^- = \{0\}$.
Thus the sum $S^+ + S^-$ is direct.  This sum is also orthogonal because, if $a \in S^+$ and $b \in S^-$, then
\[ \Trd_{D/\mathbb{Q}}(ab^\dag) = \Trd_{D/\mathbb{Q}}((ab^\dag)^\dag) = \Trd_{D/\mathbb{Q}}(ba^\dag) = -\Trd_{D/\mathbb{Q}}(ab^\dag) \]
so $\Trd_{D/\mathbb{Q}}(ab^\dag) = 0$.

For every $a \in S$, we have $\eta a^\dag \in \eta (\mathcal{O}_{F_0} R)^\dag = \mathcal{O}_{F_0} \eta R^\dag \subset \mathcal{O}_{F_0} R = S$.
Hence
\[ 2\eta a = (\eta a+\eta a^\dag) + (\eta a-\eta a^\dag) \in S^+ + S^-. \]
Thus $2\eta S \subset S^+\oplus S^-$.
Focusing on type IV, as $\textnormal{rk}_\mathbb{Z}(S)= 2d^2e$,
\begin{equation} \label{eqn:covolS+S-}
\textnormal{covol}(S^+) \textnormal{covol}(S^-) = \textnormal{covol}(S^+ \oplus S^-) \leqslant \textnormal{covol}(2\eta S) = 2^{2d^2e} \eta^{2d^2e} \textnormal{covol}(S).
\end{equation}
Here we measure covolumes in both $D_\mathbb{R}^-$ and $S^+ \otimes_{\mathbb{Z}} \mathbb{R}$ by the volume forms associated with the restriction of the inner product $\Trd_{D_\mathbb{R}/\mathbb{R}}(ab^\dag)$.

For all $a,b \in S$, $\eta ab^\dag \in S$ and so $\Trd_{D/\mathbb{Q}}(ab^\dag) \in \eta^{-1}\mathbb{Z}$.
Consequently $\textnormal{covol}(S^+) \geq \eta^{-\textnormal{rk}_\mathbb{Z}(S^+)} = \eta^{-d^2e}$ so by \eqref{eqn:disc-covol} applied to $S$ and \eqref{eqn:covolS+S-},
\[ \textnormal{covol}(S^-) \leqslant \eta^{d^2e} \cdot 2^{2d^2e} \eta^{{2d^2e}} \textnormal{covol}(S) = 2^{2d^2e} \eta^{3d^2e} \cdot d^{-d^2e} \abs{\textnormal{disc}(S)}^{1/2}. \]
(For type III with $d=2, \textnormal{rk}_\mathbb{Z}(S)= 4e, \textnormal{rk}_\mathbb{Z}(S^+)= e, \textnormal{rk}_\mathbb{Z}(S^-)= 3e$, we get $\textnormal{covol}(S^-) \leqslant 2^{2e} \eta^{5e}\abs{\textnormal{disc}(S)}^{1/2}.$)

Therefore, using \cref{dual-ideal}, \cref{L1} and \cref{GCI},
\begin{align*}
    \textnormal{covol}(JS^-)
  & = \Nm(\fc)^{2d^2} \Nm(I)^{d^2} \textnormal{covol}(S^-)
\\& \leqslant [\mathcal{O}_{F_0} : R \cap F_0]^{4d^2} \cdot (d^{-2d^2e}\abs{\textnormal{disc}(S)})^{d^2} \cdot 2^{2d^2e} \eta^{3d^2e} d^{-d^2e}\abs{\textnormal{disc}(S)}^{1/2}
\\& = 2^{2d^2e} \eta^{3d^2e} d^{-d^2e(1+2d^2)} [\mathcal{O}_{F_0} : R \cap F_0]^{4d^2} \abs{\textnormal{disc}(S)}^{1/2+d^2}.
\end{align*}
(For type III, $\textnormal{covol}(JS^-) \leqslant 2^{-10e} \eta^{5e} [\mathcal{O}_{F_0} : R \cap F_0]^{12} \abs{\textnormal{disc}(S)}^{7/2}$ ).

\noindent Applying \eqref{eqn:disc-lower-bound} to $S$, we see that $\abs{\textnormal{disc}(S)} \geq d^{2d^2e},$ we get $\abs{\textnormal{disc}(S)}^{1/2-d^2} \leq (d^{2d^2e})^{1/2-d^2}=d^{-d^2e(-1+2d^2)} $.
Using \cref{index-RcapF}, we deduce that
\begin{align*}
    \textnormal{covol}(JS^-)
  & \leqslant 2^{2d^2e} \abs{\textnormal{disc}(S)}^{1/2-d^2} \eta^{3d^2e} d^{-d^2e(1+2d^2)} [\mathcal{O}_{F_0} : R \cap F_0]^{4d^2} \abs{\textnormal{disc}(S)}^{2d^2} \\
   & \leqslant 2^{2d^2e} \eta^{3d^2e} d^{-d^2e(-1+2d^2)} d^{-d^2e(1+2d^2)} \abs{\textnormal{disc}(R)}^{2d^2} \\
   & \leqslant 2^{2d^2e} \eta^{3d^2e} d^{-4d^4e} \abs{\textnormal{disc}(R)}^{2d^2} .
\end{align*}
(For type III, $\textnormal{covol}(JS^-) \leqslant 2^{-20e} \eta^{5e} \abs{\textnormal{disc}(R)}^{6} $.)

\noindent Since $JS^-$ is a free $\mathbb{Z}$-module of rank $d^2e$, there exists $\omega \in JS^- \setminus \{0\}$ with
\[ \abs{\omega}_D \leqslant \sqrt{\gamma_{d^2e}} \textnormal{covol}(JS^-)^{1/d^2e} \leqslant \sqrt{\gamma_{d^2e}} \cdot 2^{2} \eta^3 d^{-4d^2} \abs{\textnormal{disc}(R)}^{2/e}.
 \]
Similarly repeating the same argument for type III with $d=2, \textnormal{rk}_\mathbb{Z}(S)= 4e, \textnormal{rk}_\mathbb{Z}(S^+)= e, \textnormal{rk}_\mathbb{Z}(S^-)= 3e$, we get
\begin{align*}
    \textnormal{covol}(JS^-) \leqslant 2^{-20e} \eta^{5e} \abs{\textnormal{disc}(R)}^{6}.
\end{align*}
Since $JS^-$ is a free $\mathbb{Z}$-module of rank $3e$, there exists $\omega \in JS^- \setminus \{0\}$ with
\[ \abs{\omega}_D \leqslant \sqrt{\gamma_{3e}} \textnormal{covol}(JS^-)^{1/3e} \leqslant \sqrt{\gamma_{3e}} \cdot 2^{-20/3} \eta^{5/3} \abs{\textnormal{disc}(R)}^{6/e}.
\qedhere \]
\end{proof}

\begin{lemma}\label{Rs}
For any $a\in R$, the element $d\,\Nrd_{D/\mathbb{Q}}(a)^d\, a^{-1}$ belongs to $R^*$.
\end{lemma}

\begin{proof} 
The left multiplication map $L_a:D\to D$, $b\mapsto ab$, has trace
$\Tr_{D/\mathbb{Q}}(a)=\Tr(L_a)$.
Since $a,b\in R$ imply $L_a,L_b\in M_n(\mathbb{Z})$, the matrix
$\det(L_a)L_a^{-1}$ also has integer entries.  
Thus $\Tr_{D/\mathbb{Q}}(\Nrd(a)a^{-1}b)\in\mathbb{Z}$ for all $b\in R$,
and since $\Tr_{D/\mathbb{Q}}=d\,\Trd_{D/\mathbb{Q}}$ the claim follows.
\end{proof}

\begin{lemma}\label{w1w2}
Let $D$ be of type IV.  
For $w_1,w_2\in R$ and $\omega$ as in \cref{as1}, define
\[
b = d^{2-8d^4e^2}\abs{\disc(R)}^{4d^2e}
        \Nrd(\psi(w_1,w_2))^{2d}\,
        \omega\,\psi(w_2,w_1)^{-1},
\qquad
a = b\,\psi(w_2,w_2)\,\psi(w_1,w_2)^{-1}.
\]
Then $a,b\in R$.  
The same statement holds for type III with
\[
b = 2^{2-32e^2}\abs{\disc(R)}^{8e}
        \Nrd(\psi(w_1,w_2))^{4}\,
        \omega\,\psi(w_2,w_1)^{-1}.
\]
\end{lemma}

\begin{proof}
    Since $\Trd_{D/\mathbb{Q}} \psi(L \times L) \subset \mathbb{Z}$, $\psi(L \times L) \subset R^*$, so $d^{-2d^2e} \abs{\disc(R)} \psi(w_1,w_2) \in R$. Let $x=d^{-2d^2e} \abs{\disc(R)}$ for type IV and $x=2^{-4e}\abs{\disc(R)}$ with $d=2$ for type III. Now applying previous lemma, 
    \begin{align*}
        xd (\Nrd_{D/\mathbb{Q}}(x \psi(w_1,w_2)))^d (x \psi(w_1,w_2))^{-1} \in R.
    \end{align*}
    So as $(\Nrd_{D/\mathbb{Q}}(x \psi(w_1,w_1)))^d = (x)^{2d^2e}(\Nrd_{D/\mathbb{Q}}\psi(w_1,w_2))^d$ for type IV whereas it's $= (x)^{4e}(\Nrd_{D/\mathbb{Q}}\psi(w_1,w_2))^2$ for type III,
    \begin{align*}
        dx^{2d^2e}(\Nrd_{D/\mathbb{Q}}\psi(w_2,w_1))^d\psi(w_2,w_1))^{-1} \in R
    \end{align*}
    for type IV and exponent of $x$ replaced by $4e$ with $d=2$ for type III.
    
    \noindent By using \cref{as1}(ii) and \cref{Rs}, $ \omega d (\Nrd_{D/\mathbb{Q}}(x \psi(w_2,w_1)))^d (x \psi(w_2,w_1))^{-1} \in R$, i.e
    \begin{align*}
        dx^{-1}x^{2d^2e}(\Nrd_{D/\mathbb{Q}}\psi(w_2,w_1))^d \omega (\psi(w_2,w_1))^{-1} \in R. 
    \end{align*}
    As $\Nrd_{D/\mathbb{Q}} (\psi(w_1,w_2)) =(\Nrd_{D/\mathbb{Q}} (\psi(w_2,w_1))$ and $x \psi(w_2,w_2) \in R$, we have $a,b \in R$.
\end{proof}

The following lemma is useful for calculating discriminants of skew-Hermitian forms. This lemma is adapted from {\citep[Subsection 3.C]{Lattices}} which still holds for types III and IV.

\begin{lemma} \label{trangular}
Let $(D, \dag)$ be a division $\mathbb{Q}$-algebra with a positive involution of type III or IV.
Let $V$ be a left $D$-vector space with a non-degenerate $(D,\dag)$-skew-Hermitian form $\psi : V \times V \to D$.
Let $R$ be an order in $D$.
Let $v_1, \dotsc, v_m$ be a $D$-basis for $V$.
Suppose that for all  $i = 1, \dotsc, m$, there is exactly one permutation $\sigma \in S_m$  for which $\psi(v_i, v_{\sigma(i)}) \neq 0$ . Let $k=2^{-4e}$ for type III and $k=d^{-2d^2e}$ for type IV.
Then
\[ \abs{\textnormal{disc}(Rv_1 + \dotsb + Rv_m, \Trd_{D/\mathbb{Q}} \psi)} = k \abs{\textnormal{disc}(R)}^m \prod_{i=1}^m \abs{\Nm_{D/\mathbb{Q}}(\psi(v_i, v_{\sigma(i)}))}. \]
\end{lemma}

The following lemma is adapted to types III and IV division algebra from {\citep[Lemma 3.9]{Lattices}}
\begin{lemma} \label{good-norm}
Let $(D,\dag)$ be a type III or IV algebra with positive involution.
Let $V$ be a left $D$-vector space of dimension $m$, equipped with a non-degenerate $(D,\dag)$-skew-Hermitian form $\psi : V \times V \to D$.
Then there exists a $D$-norm $\abs{\cdot}$ on $V_\mathbb{R}$ which is adapted to $\psi$.
\end{lemma}

\section{Cases of Zilber-Pink}\label{cases-of-ZP}

In this section, we present a proof of \cref{zp1}. The arguments are inspired by the strategy in \citep[sec.~5--8]{Lattices}, and we adopt notation from \citep[sec.~2.2, 2.4]{Orr18}. Our approach relies on the Pila-Zannier method, originally introduced to give an alternative proof of the Manin-Mumford conjecture \citep{PZ08}. 

This section generalise results from \citep[sections 5--8]{Lattices} to encompass PEL types III and IV.

\subsection{Shimura data}\label{Shimura-date}

Define
\[
J_2 = \begin{pmatrix}0 & 1 \\ -1 & 0\end{pmatrix}, \qquad
J_n = J_2^{\oplus\, n/2}, \quad \text{for even } n>0.
\]
For any ring $R$, the \emph{symplectic group} is
\[
\Sp_{2g}(R) = \{ T \in M_{2g}(R) : T J_{2g} T^\mathsf{t} = J_{2g} \}.
\]

The \emph{Siegel upper half-space} is
\[
\mathcal{H}_g = \{ Z \in M_g(\CC) : Z^\mathsf{t} = Z, \Im(Z) > 0 \},
\]
which can be identified with $\CC^{g(g+1)/2}$. The group $\Sp_{2g}(\RR)$ acts on $\mathcal{H}_g$ via
\[
\begin{pmatrix}A & B \\ C & D\end{pmatrix} \cdot Z = (AZ + B)(CZ + D)^{-1}.
\]

The quotient $\Sp_{2g}(\ZZ)\backslash \mathcal{H}_g$ corresponds to the moduli space $\mathcal{A}_g$, with the projection
\[
\pi : \mathcal{H}_g \longrightarrow \mathcal{A}_g.
\]

The \emph{Deligne torus} $\mathbb{S} = \mathrm{Res}_{\RR}^{\CC}(\mathbb{G}_m)$ satisfies
\[
\mathbb{S} = \left\{
\begin{pmatrix} a & b \\ -b & a \end{pmatrix} : a^2 + b^2 \neq 0
\right\}, \quad \mathbb{S}(\RR) \cong \CC^\times.
\]

\begin{definition}
A \emph{Shimura datum component} is a pair $(\mathcal{G}, X^+)$ with $\mathcal{G}/\QQ$ reductive and $X^+$ a connected component of a $\mathcal{G}(\RR)$-conjugacy class of homomorphisms $h: \mathbb{S} \to \mathcal{G}_\RR$ satisfying:
\begin{itemize}
    \item The action of $\mathbb{S}$ on $\Lie(\mathcal{G})_\CC$ has weights in $\{(0,0),(1,-1),(-1,1)\}$.
    \item $\mathrm{ad}(h(i))$ induces a Cartan involution on $\mathcal{G}^{\mathrm{ad}}_\RR$.
    \item No simple $\QQ$-factor of $\mathcal{G}^{\mathrm{ad}}$ becomes compact over $\RR$.
\end{itemize}
\end{definition}

If $\mathcal{H} \subset \mathcal{G}$ and $X_\mathcal{H} \subset X_\mathcal{G}$, then $(\mathcal{H}, X_\mathcal{H})$ is a \textbf{Shimura subdatum}. A \textbf{Shimura variety component} is $\Gamma\backslash X^+$, where $\Gamma \subset \mathcal{G}(\QQ)$ is a congruence subgroup stabilizing $X^+$.

\begin{definition}
A subvariety $V \subset \Gamma_\mathcal{G}\backslash X_\mathcal{G}^+$ is \textbf{special} if there exists a morphism of Shimura variety components
\[
f: \Gamma_\mathcal{H}\backslash X_\mathcal{H}^+ \to \Gamma_\mathcal{G}\backslash X_\mathcal{G}^+
\]
with $f(\Gamma_\mathcal{H}\backslash X_\mathcal{H}^+) = V$. A \textbf{special point} is a $0$-dimensional special subvariety.
\end{definition}

Let $L = \ZZ^{2g}$ and $V = L_\QQ$, equipped with the standard symplectic form $\phi$ defined by $J_{2g}$. Denote $\mathcal{G} = \mathcal{G}Sp(V, \phi)$ and $\Gamma = Sp_{2g}(\ZZ)$. Let $X^+$ be the $\mathcal{G}(\RR)^+$-conjugacy class of
\begin{equation} \label{eqn:h0}
h_0(a+ib) \mapsto \begin{pmatrix} a & b \\ -b & a \end{pmatrix}^{\oplus g}.
\end{equation}
Then $(\mathcal{G}, X^+)$ is a Shimura datum component and $X^+ \simeq \mathcal{H}_g$. The moduli space $\mathcal{A}_g$ is the corresponding Shimura variety.

For a PEL-type special subvariety $S \subset \mathcal{A}_g$, let $R$ be its generic endomorphism ring. Fix $x \in X^+$ with image $s \in S(\CC)$ generic. Let $\mathcal{H}$ be the centralizer of $R$ in $\mathcal{G}$. Then $\mathcal{H}^\circ$ is called the \emph{general Lefschetz group} of $S$.

Note that in the simple type IV case, $\mathcal{H} = \mathcal{H}^\circ$, whereas in the simple type III case, this equality does not necessarily hold; however, this distinction will not be relevant for our purposes.

We now state the following lemmas and theorems, which will be used in the proof of \cref{zp1}.

\begin{lemma}\label{conj-class-datum}\citep[Lemma 3.7]{UY13} 
Let $(\mathcal{G}, X)$ be a Shimura datum and $\mathcal{H} \subseteq \mathcal{G}$ a subgroup. Then there are only finitely many $X_\mathcal{H}$ such that $(\mathcal{H}, X_\mathcal{H})$ forms a Shimura subdatum.
\end{lemma}

\begin{definition}
The conjugacy class of a sub-Shimura datum $(\mathcal{H}_0, X_0^+)$ in $(\mathcal{G}, X^+)$ is
\[
\{ (g\mathcal{H}_0g^{-1}, g X_0^+) : g \in \mathcal{G}(\RR)^+, g\mathcal{H}_0g^{-1} \text{ defined over } \QQ \}.
\]
\end{definition}
Note by \cref{conj-class-datum}, there are only finitely many conjugacy classes for given $\mathcal{H}_0$. 

\begin{proposition}\label{codim-pel}
Let $S \subset \mathcal{A}_g$ be a proper PEL-type special subvariety. Then $\dim(S) \leq \dim(\mathcal{A}_g) - g + 1$.
\end{proposition}

\subsection{Conjugacy classes of simple PEL-III and IV type}\label{shimmer}

The Shimura subdatum components of $(\mathcal{G}, X^+)$ of simple PEL-III and IV type lie in only finitely many $\mathcal{G}(\bR)^+$-conjugacy classes. Let $d$, $e$, $m$ be positive integers such that $kem = 2g$, where $k=4$ for PEL-III type with unique signature $\textbf{s}$ and $k=2d^2$ for PEL-IV type and $\textbf{s}$ be a signature. (Note we don't have to worry about signature issues for PEL-III type case.)
For fixed $g$, there are only finitely many integers $d,e,m,\textbf{s}$ satisfying these conditions. As we shall show, each quadraple $d,e,m,\textbf{s}$ corresponds to a finitely many $\mathcal{G}(\bR)^+$-conjugacy class of Shimura subdatum components of simple PEL III/IV type.

For type III, Let $\mathbb{Q}(i,j,k)$ be the Hamilton quaternion algebra over $\mathbb{Q}$ and let $D_0 = (\mathbb{Q}(i,j,k))^e$.
Define a $\mathbb{Q}$-algebra homomorphism $\iota_0 \colon D_0 \to \rM_{2g}(\mathbb{Q})$ as follows:
\[ \iota_0( A_1, \cdots,A_e  )
= \bigoplus_{i=1}^{e} \big((
\tilde{A}_i)\big) \]
where $\tilde{A_i}$ is a image of homomorphism map $\varphi: \mathbb{Q}(i,j,k) \rightarrow \text{M}_{4}(\mathbb{Q})$ of $A_i$, we will map the quaternion units $1, i, j, k$ to $4 \times 4$ matrices.

\begin{align*}
\varphi(1) = \begin{pmatrix}
1 & 0 & 0 & 0 \\
0 & 1 & 0 & 0 \\
0 & 0 & 1 & 0 \\
0 & 0 & 0 & 1
\end{pmatrix},
&\varphi(i) = \begin{pmatrix}
0 & 1 & 0 & 0 \\
-1 & 0 & 0 & 0 \\
0 & 0 & 0 & 1 \\
0 & 0 & -1 & 0
\end{pmatrix},
\varphi(j) = \begin{pmatrix}
0 & 0 & 1 & 0 \\
0 & 0 & 0 & -1 \\
-1 & 0 & 0 & 0 \\
0 & 1 & 0 & 0
\end{pmatrix},
\\& \varphi(k) = \begin{pmatrix}
0 & 0 & 0 & 1 \\
0 & 0 & 1 & 0 \\
0 & -1 & 0 & 0 \\
-1 & 0 & 0 & 0
\end{pmatrix}.   
\end{align*}

We define
\[
\varphi(A_i) = a \varphi(1) + b \varphi(i) + c \varphi(j) + d \varphi(k)
\]

For type IV, let $D_0 = \rM_d(\mathbb{Q}(i))^e$.
Define a $\mathbb{Q}$-algebra homomorphism $\iota_0 \colon D_0 \to \rM_{2g}(\mathbb{Q})$ as follows:
\[ \iota_0( A_1, \cdots,A_e  )
= \bigoplus_{i=1}^{e} \Bigg(\begin{pmatrix}
\tilde{A}_i
\end{pmatrix} \oplus \cdots \oplus \begin{pmatrix}
\tilde{\bar{A}}_i
\end{pmatrix} \oplus \begin{pmatrix}
\tilde{A}_i
\end{pmatrix} \oplus \cdots \oplus \begin{pmatrix}
\tilde{\bar{A}}_i
\end{pmatrix}  \Bigg) \]
where copy of $\begin{pmatrix}
\tilde{A}_i
\end{pmatrix}$ occurs $r_i$ times and $\begin{pmatrix}
\tilde{\bar{A}}_i
\end{pmatrix}$ occurs $md-r_i$ times, where $r_i$ is a signature, where $\tilde{X}$ just outputs for each entry $x_{ij}+iy_{ij}$ of matrix $X$, $\begin{pmatrix}
x_{ij} & -y_{ij} \\
y_{ij}& x_{ij}
\end{pmatrix} $.

\textbf{Note:} The rest of the section works with PEL-III type with only slight adjustment in parameter so we only focus on PEL-IV type. As mentioned before we don't have to worry about signature issues for PEL-III type

We view $V$ as a left $D_0$-module via $\iota_0$. Let $t$ denote the involution of $D_0$ which is complex conjugate-transpose on each factor.
Since $2dem$ is even, routine calculation shows that $\iota_0(D_0)$ commutes with $J_{2g}$ and so, for all $a \in D_0$ and $x, y \in V$, we have
\[ \phi(ax, y) = x^t \iota(a)^t J_{2g} y = x^t J_{2g} \iota(a)^t y = \phi(x, a^t y). \]
Thus $\phi \colon V \times V \to \mathbb{Q}$ is a $(D_0,t)$-compatible symplectic form.
By \citep[Corollary 3.3]{Lattices}, there is a unique non-degenerate $(D_0,t)$-skew-Hermitian form $\psi_0 \colon V \times V \to D_0$ such that $\phi = \Trd_{D_0/\mathbb{Q}} \psi_0$. (Routine calculation shows that $\phi = \Trd_{D_0/\mathbb{Q}} \psi_0$ for PEL-III type as well). 

Let $\mathcal{H}_0$ denote the centraliser of $\iota_0(D_0)$ in $\mathcal{G}$.

\begin{lemma}\label{conjclass1}
Let $(\mathcal{H}, X_\mathcal{H}^+) \subset (\mathcal{G}Sp_{2g}, \mathcal{H}_g)$ be a Shimura subdatum component of PEL-III/IV type. Let $D$ be its generic endomorphism algebra with center $F$. Then $\mathcal{H}_\RR$ is $\mathcal{G}(\RR)^+$-conjugate to the group $\mathcal{H}_0$ defined as the centralizer of $\iota_0(D_0)$, for
\[
d = \sqrt{\dim_F(D)}, \quad ke = [F:\QQ], \quad m = 2g / (kd^2 e),
\]
where $d = 2$, $k=1$ for type III and $k=2$ for type IV.
\end{lemma}

The same proof of \citep[Lemma 5.1]{Lattices} works for above lemma, with the fixed signature $\textbf{s}$ for the PEL-IV type.


We now use the representation that comes from \citep[Sec. 6]{Lattices} as it was already made generic to be able to apply for PEL-IV type (also PEL-III type) case . In our application to \cref{zp1}, $\mathcal{H}_0$ shall be equal to the group $\mathcal{H}_0$ defined be the group defined just before \cref{conjclass1}. Define $D_0$ and $\iota_0 \colon D_0 \to \rM_n(\mathbb{Q})$ as in section~\ref{shimmer} (with $n=2g$).
Let $E_0$ be the centraliser of $\iota_0(D_0)$ in $\rM_n(\mathbb{Q})$, so we have 
\begin{equation} \label{eqn:H0}
\mathcal{H}_0 = \mathcal{G}Sp \cap E_0.
\end{equation}

We can do the similar calculation as in the proof of \citep[Prop 6.1]{Lattices} after extending scalars to $\mathbb{C}$, and routine calculation shows that we have desired condition:  $E_{0,\mathbb{C}} \cong \rM_{md}(\mathbb{C})^{2e}$ and the resulting $E_{0,\mathbb{C}}$-module structure on $V_\mathbb{C}$ is isomorphic to the direct sum of $d$ copies of each of the $e$ irreducible representations of $E_{0,\mathbb{C}}$.

\citep[Prop~7.1]{Lattices} holds for $\mathcal{H}_0$ defined in \cref{eqn:H0}. The argument generalises \citep[section~5.5]{DO21b} and  \cref{height-bound} plays the role of \citep[Lemma~5.7]{DO21b} 

By \citep[Prop 6.1]{Lattices} and generalised \citep[Prop~7.1]{Lattices} gives us the following theorem:
\begin{theorem}\label{IMP}
 Let $d$, $e$ and $m$ be positive integers.
Let $n = kem$, where $k=4$ for PEL-III type and $k=2d^2$ for PEL-IV type.
Let $L = \bZ^n$ and let $\phi \colon L \times L \to \bZ$ be the standard symplectic form as in section~\ref{Shimura-date}.
Let $\gG = \mathcal{G}Sp(L_\mathbb{Q}, \phi) = \mathcal{G}Sp_{n,\mathbb{Q}}$ and let $\Gamma = \gSp_n(\bZ)$.
Let $\mathcal{H}_0$ be the subgroup of $\gG$ defined in~\eqref{eqn:H0}.
Then there exists a finitely generated, free $\ZZ$--module $\Lambda$, a representation $\rho_L:\gG\to\gGL(\Lambda_\QQ)$ such that $\Lambda$ is stabilised by $\rho_L(\Gamma)$, a vector $w_0\in\Lambda$ and positive constants $\newC{rm1}$, $\newC{re1}$ such that:
\begin{enumerate}[(i)]
\item $\Stab_{\gG,\rho_L}(w_0) = \mathcal{H}_0$;

\item the orbit $\rho_L(\gG(\bR))w_0$ is closed in $\Lambda_\bR$;

\item for each $u \in \gG(\bR)$, if the group $\mathcal{H}_u = u \mathcal{H}_{0,\bR} u^{-1}$ is defined over~$\mathbb{Q}$ and $L_\mathbb{Q}$ is irreducible as a representation of $\mathcal{H}_u$ over $\mathbb{Q}$, then there exists $w_u \in \Aut_{\rho_L(\gG)}(\Lambda_\bR) w_0$ such that $\rho_L(u) w_u \in \Lambda$ and
\[ \abs{w_u} \leq \refC{rm1} \abs{\disc(R_u)}^{\refC{re1}}, \]
where $R_u$ denotes the ring $\End_{\mathcal{H}_u}(L) \subset \rM_{2g}(\bZ)$.
\end{enumerate}   
\end{theorem}

\subsection{Proof of Theorem \ref{zp1}}\label{FP1}

In fact, instead of proving \cref{zp1}, we will prove the following, more general theorem. (Recall that, by \cref{codim-pel}, for $g\geq 3$, all proper special subvarieties of PEL type of $\cA_g$ have codimension at least~$2$.)

\begin{theorem}\label{ZP-end} 
Let $g\geq 3$ and let $C$ be an irreducible algebraic curve in $\cA_g$. Let $S$ denote the smallest special subvariety of $\cA_g$ containing $C$.
Let $\Omega$ denote the set of special subvarieties of $\cA_g$ of simple PEL type III or IV of dimension at most $\dim(S)-2$. Let $\Sigma$ denote the set of points in $\cA_g(\CC)$ which are endormorphism generic in some $Z\in\Omega$.
If $C$ satisfies Conjecture \ref{LGO-general}, then $C\cap\Sigma$ is finite.
\end{theorem}

Conjecture \ref{LGO-general} is the natural generalisation of Conjecture \ref{GO0}.

\begin{conjecture}\label{LGO-general}
Let $C$ and $\Sigma$ be as in Theorem \ref{ZP-end} and let $L$ be a finitely generated subfield of $\CC$ over which $C$ is defined.
Then there exist positive constants $\newC{ZPem}$ and $\newC{ZPee2}$ such that
\begin{align*}
    \#\Aut(\CC/L)\cdot s\geq\refC{ZPem}|\disc(\End(A_s))|^{\refC{ZPee2}}
\end{align*}
for all $s\in C\cap\Sigma$.
\end{conjecture}

If we denote by $\pi:X^+\to\cA_g$ the uniformising map, then $\pi|_{\cF}$ is definable in the o-minimal structure $\RR_{\rm an,exp}$ for some  $\cF$ a fundamental set in $X^+$. (see \citep{Lattices} for details).

Let $\Sigma$ be the union of set $\Sigma_{d,e,m,\textbf{s}}$, where $d$, $e$, $m$ are positive integers satisfying $d^2em=2g$, and $\textbf{s}$ denote the particular signature (only one signature for type III).
Since there are only finitely many choices for such $d$, $e$, $m$ and $\textbf{s}$ (given~$g$), in order to prove \cref{ZP-end}, it suffices to prove that $C \cap \Sigma_{d,e,m,\textbf{s}}$ is finite for each $d$, $e$, $m$, $\textbf{s}$.

From now on, we fix such integers $d$, $e$ $m$ and $\textbf{s}$.  Let $\mathcal{H} \subset \gG$ be the group defined by \cref{eqn:H0} associated with these parameters.
Let $X_i^+ = \mathcal{H}(\bR)^+h_i$, so that $(\mathcal{H}, X_i^+)$ is the Shimura subdatum component of $(\gG, X^+)$ given by finitely many choices of \cref{conj-class-datum}.

Following \cref{IMP}, we can use \citep[Theorem 1.2]{DO21b}, to show that there exist positive constants $\newC{bm21}$ and $\newC{ber1}$ with the following property: for every $u\in\gG(\RR)$ and $w_u\in \Aut_{\rho_L(\gG)}(\Lambda_\bR) w_0$ such that $\mathcal{H}_u = u \mathcal{H}_{0,\bR} u^{-1}$ is defined over~$\mathbb{Q}$ and $\rho_L(u) w_u \in \Lambda$, there exists a fundamental set for $\Gamma\cap\mathcal{H}_u(\RR)$ in $\mathcal{H}_u(\RR)$ of the form 
\[ B_u\cF_\gG u^{-1}\cap\mathcal{H}_u(\RR),\]
where $B_u\subset\Gamma$ is a finite set such that
\[\abs{\rho_L(b^{-1}u)w_u} \leq \refC{bm21} \abs{w_u}^{\refC{ber1}}\]
for every $b\in B_u$.

\medskip

For any $w\in\Lambda_\RR$, we write $\gG(w)$ for the real algebraic group $\Stab_{\gG_\RR,\rho_L}(w)$. Fixing a basis for $\Lambda$, we may refer to the height $\rH(w)$ of any $w\in\Lambda$ (namely, the maximum of the absolute values of its coordinates with respect to this basis.)

Pick $X_i^{+}$ such that $(\mathcal{H}, {X_i}^{+})$ be Shimura datum generating one of the special subvariety in $\Omega$. Pick $h_i$ be an element of ${X_i}^+$
\begin{lemma} \label{z-w}
Let $P\in\Sigma_{d,e,m,\textbf{s}}$. There exists $z\in \pi^{-1}(P)\cap\cF$ and
\[w\in \Aut_{\rho_L(\gG)}(\Lambda_\bR)\rho_L(\gG(\RR)^+)w_0\cap\Lambda\] 
such that $z(\mathbb{S})\subset\gG(w)$ and
\[\rH(w)\leq \refC{bm21}\refC{rm1}^{\refC{ber1}} \abs{\disc(R)}^{\refC{re1}\refC{ber1}},\]
where $R=\End(A_P)\cong\End_{\gG(w)}(L)\subset \rM_{2g}(\ZZ)$.
\end{lemma}

Proof of this lemma is same as the proof of \citep[Lemma 8.3] {Lattices} with the given fixed signature  $\textbf{s}$.

We can now deduce,

\begin{corollary}\label{prop:finitely-many}
Define $\Sigma \subset \Ag$ as in \cref{zp1}.
For each $b\in\RR$, the points $s\in\Sigma$ such that $\abs{\disc(\End(A_s))} \leq b$ belong to only finitely many proper special subvarieties of PEL IV type.
\end{corollary}

\begin{proof}
Letting $s=P \in \Sigma_{d,e,m,\textbf{s}}$, and, applying \cref{z-w}, we obtain $w\in\Lambda$ satisfying
\begin{align*}
\rH(w)\leq \refC{bm21}\refC{rm1}^{\refC{ber1}} \abs{\disc(R)}^{\refC{re1}\refC{ber1}}
\end{align*}
such that $Z$ is the image in $\cA_g$ of an orbit of $\gG(w)^\der(\RR)$. By \cref{conj-class-datum}, there are only finitely many Shimura subdatum of $(\gG,X)$ associated with $\gG(w)$ and so the result follows.
\end{proof}

\subsection{Proof of Theorem \ref{ZP-end}}\label{FP}
We now apply Pila-Zannier strategy, where Parameter Height bound and LGO are the arithmetic ingredients.


Let $L$ be a finitely generated field of definition for $C$. 
Let $P\in C\cap\Sigma_{d,e,m,\textbf{s}}$, then $P$ is endomorphism generic in some special subvariety, $ Z \in \Omega$ (where $\Omega$ is defined in Theorem \ref{ZP-end}). Then $Z$ is an irreducible component of $\cM_R$, where $R = \End(A_P)$. 
Since $\cM_R$ is $\Aut(\CC)$-invariant and its (analytic) irreducible components are algebraic subvarieties of $\Ag$, for each $\sigma \in \Aut(\CC)$, $\sigma(Z)$ is also an irreducible component of $\cM_R$.
Thus, $\sigma(Z)$ is also a special subvariety of PEL III/IV type with the same parameters $d,e,m$. Furthermore, $\dim(\sigma(Z)) = \dim(Z)$, so $\sigma(Z) \in \Omega$. Since $\End(A_{\sigma(P)}) \cong \End(A_P)$, $\sigma(P)$ is endomorphism generic in $\sigma(Z)$, so $\sigma(P) \in \Sigma_{d,e,m}$.

In this paper \textit{definable} will mean \textit{definable
in the o-minimal structure $\mathbb{R}_{\textnormal{an}, \exp}$.}
See \citep{Pila11} for an introduction o-minimality and
\citep{DM94, DMM94} for the properties of $\mathbb{R}_{\textnormal{an}, \exp}$.
Let $\mathcal{C}=\pi^{-1}(C)\cap\mathcal{F}$ -- a set definable in the o-minimal structure $\RR_{\rm an,exp}$ (see \citep{KUY16} for more details).  Let say there are $r$ (finite) different possible signatures.
By pigeonhole principle, we have at least one $\textbf{s}$ such that $\#(\Aut(\CC/L) \cdot\sigma(P) \cap \Sigma_{d,e,m,\textbf{s}}) \geqslant 1/r \cdot \#\Aut(\CC/L)\cdot\sigma(P)$. Fix such $\textbf{s}$, so we obtain points $\sigma(P)\in C\cap\Sigma_{d,e,m,\textbf{s}}$ and, for each $\sigma$, we let $z_\sigma\in\mathcal{F}\cap\pi^{-1}(\sigma(P))$ and we let $W=\Aut_{\rho_{L}(\gG)}(\Lambda_\RR)\rho_{L}(\gG(\RR)^{+})w_0$ and $w_\sigma\in W \cap\Lambda$ be the elements afforded to us by Proposition \ref{z-w}. That is, $z_\sigma(\mathbb{S})\subset \gG(w_\sigma)$ and 
\begin{align*}
\rH(w)\leq \refC{bm21}\refC{rm1}^{\refC{ber1}} \abs{\disc(R)}^{\refC{re1}\refC{ber1}},
\end{align*}

Note that we also have $z_\sigma\in\mathcal{C}$.

We obtain a set $\Theta$ of tuples $(w_\sigma,z_\sigma)\in\Lambda\times\mathcal{C}$ belonging to the definable set
\[D=\{(w,z)\in \Lambda_\RR\times\mathcal{C}: w\in W,\ z(\mathbb{S})\subset \gG(w)\}.\]

Let $\pi_1:D\rightarrow\Lambda_\RR$ and let $\pi_2:D\rightarrow\mathcal{C}$ denote the projection maps. By \cref{LGO-general}, for a positive constant $\newC{alphaaa}$ we have
\begin{align*}
A:=\#\pi_2(\Theta) \geq 1/r \cdot \#\Aut(\CC/L)\cdot P & =  1/r \cdot \#\Aut(\CC/L)\cdot\sigma(P)
\\& \geq  1/r \cdot \refC{ZPem} \abs{\disc(R_\sigma)}^{\refC{ZPee2}} \geq \refC{alphaaa} H(w_\sigma)^{\refC{ZPee2}/{\refC{re1}\refC{ber1}}}.
\end{align*}

Now we can use essentially the same proof as in \citep[Corollary 8.4]{Lattices}, incorporating details from \citep[Corollary 6.4]{DO21b} with 
\begin{align}\label{ctspath}
    (g^{-1}_tz_t)(\mathbb{S})\subset g^{-1}_t \gG(v_t)g_t=\gG(g^{-1}_tw_t)=\gG(w_0) = \mathcal{H}(\RR).
\end{align}
Using \cref{conj-class-datum} and the fact that all of the conjugacy classes are disjoint from one another, the continous path $g^{-1}_tz_t$ lies on the unique pre-special subvariety of $\Hg$ associated with $\mathcal{H}(\RR)$, in one of the conjugacy classes of, ${X_i}^+$.

\subsection{Proof of Theorem \ref{cor2}}
We also have the following result due to Fite-Goodman \citep{FG25}.
\begin{theorem}\label{FG} \citep[Cor 1.2]{FG25}
    Let $A$ be an abelian threefold defined over number field $K$ such that $\End^0(A)=M$, where $M$ is an imaginary qudratic field. Then $M$ has class number, $\textnormal{cl}(M)$ $\leq [K:K \cap M]$.
\end{theorem}

\begin{corollary}\label{impcor}
Let $ \Sigma $ denote the set of points $ s \in \mathcal{A}_3(\mathbb{C}) $ for which the endomorphism ring of the associated abelian variety $ A_s $ is $ \mathcal{O}_M $, where $ M $ is an imaginary quadratic field. For any $ s \in \Sigma $, there exist constants $ c_1, c_2 > 0 $ such that
\[
[\mathbb{Q}(s) : \mathbb{Q}] \geq c_1  \abs{\disc(\End(A_s)) }^{c_2}.
\]
\end{corollary}

\begin{proof}
Let $ \mathfrak{A} \to   \mathcal{A}_{3} $ be an abelian scheme. This abelian scheme induces a morphism of varieties $
q :  \mathcal{A}_{3,3}  \to \mathcal{A}_3,$ 
defined over $ \mathbb{Q} $. Let $ \widetilde{\mathcal{A}_3} $ be an irreducible component of the preimage of $\mathcal{A}_3$ in $ \mathcal{A}_{3,3} $, defined over a finite extension $ \widetilde{L} $ of $\mathbb{Q}$. The universal abelian scheme over $ \mathcal{A}_{3,3} $ restricts to an abelian scheme $\widetilde{\mathfrak{A}} \to \widetilde{\mathcal{A}_3}.$ If $ s \in \Sigma $, then there exists a point $ \widetilde{s} \in \widetilde{\mathcal{A}_3} $ such that $ q(\widetilde{s}) = s $. Thanks to \cref{FG} and Siegel's theorem, for any $ \varepsilon > 0 $, there exists a constant $ c(\varepsilon) > 0 $ such that
\[
[\widetilde{L}(s) : \widetilde{L}(s) \cap M] \geq \operatorname{cl}(M) \geq c(\varepsilon) \left| \operatorname{disc}(M) \right|^{1/2 - \varepsilon} = c(\varepsilon)\abs{\disc(\mathcal{O}_M)}^{1/2 - \varepsilon}.
\]

In other words, $[\mathbb{Q}(\widetilde{s}) : \mathbb{Q}] \geq c(\varepsilon)/2 \abs{\disc(\mathcal{O}_M)}^{1/2 - \varepsilon}.$ Since $q$ is a finite morphism, there exists $ c_k > 0 $ such that
\[
[\mathbb{Q}(s) : \mathbb{Q}] \geq c_k \cdot [\mathbb{Q}(\widetilde{s}) : \mathbb{Q}].
\]
\end{proof}

\section{Proof of Theorem~\ref{height-bound}} \label{sec:height bound}

In this section we establish our main height estimate for weakly unitary bases
associated with skew-Hermitian forms.  
The argument proceeds by an adaptation of the Gram--Schmidt procedure and follows
the same inductive strategy as in \citep[Section~4]{Lattices}, though the presentation
here is self-contained.

At first glance, one might try to pick a shortest nonzero vector
$v_1\in V$ with $\psi(v_1,v_1)\neq 0$; however, as explained in
\citep[Section~4]{Lattices}, such a vector may not always exist.
Thus we begin by selecting either:
\begin{enumerate}[(1)]
\item a short vector $v_1\in V$ for which $\psi(v_1,v_1)\neq 0$, or
\item a pair of short vectors $v_1,v_2\in V$ such that $\psi$ restricts to a
      nondegenerate form on $Dv_1+Dv_2$, and such that $v_1,v_2$ already form a
      weakly unitary basis of this $D$-subspace.
\end{enumerate}

The next lemma is the key device used to identify such short vectors.

\begin{lemma}[{\citep[Lemma~4.3]{Lattices}}]\label{short1}
Let $(D,\dag)$ be a division algebra over $\mathbb{Q}$ with positive involution, and let
$V$ be a left $D$-module of dimension $m$ equipped with a nondegenerate
$(D,\dag)$-skew-Hermitian form $\psi$.  
Assume also that $\abs{\cdot}$ is a $D$-norm on $V_\mathbb{R}$, and let
$w_1,\dots,w_m$ be a $D$-basis.  
Then there exist indices $i,j\in\{1,\dots,m\}$ such that:
\begin{enumerate}[(i)]
\item $\abs{w_i}\abs{w_j}\le (\abs{w_1}\cdots\abs{w_m})^{2/m}$,
\item $\psi(w_i,w_j)\neq 0$, and
\item $\psi(w_i,w_i)=0$ whenever $i\neq j$.
\end{enumerate}
\end{lemma}

We now apply this selection procedure to the lattice $L$.

\begin{lemma}\label{pre-induction}
Let $(D,\dag)$ be a division algebra of type~III or~IV endowed with a positive
involution.  
Let $V$ be a left $D$-vector space equipped with a non\-degenerate
$(D,\dag)$-skew-Hermitian form~$\psi$.  
Let $L\subset V$ be a $\mathbb{Z}$-lattice satisfying
$\Trd_{D/\mathbb{Q}}\psi(L,L)\subset\mathbb{Z}$.  
Let $R\subset\Stab_D(L)$ be an order and choose $\eta\in\mathbb{Z}_{>0}$ such that
$\eta R^\dag\subset R$.  

Then an $R$-submodule $M\subset L$ may be found with the following properties:
\begin{enumerate}[(i)]
\item $r:=\dim_D(D\otimes_R M)\in\{1,2\}$;
\item the restriction $\psi|_M$ is nondegenerate;
\item 
$\abs{\textnormal{disc}(M)}  \leqslant (\gamma^2_{d^2em}/d^3e)^{d^2er/2} \cdot \abs{\textnormal{disc}(R)}^r \cdot  \abs{\textnormal{disc}(L)}^{r/m}$ with $d=2$ for type III and \\
$\abs{\textnormal{disc}(M)}  \leqslant (\gamma^2_{2d^2em}/d^3e)^{d^2er} \cdot \abs{\textnormal{disc}(R)}^r \cdot  \abs{\textnormal{disc}(L)}^{r/m}$ for type IV; 

\item and one of the following holds:
\begin{enumerate}[(a)]

\item $D$ has Type $III$,  $r=1$ and $M = Rv_1$ for some $v_1$ such that
 \[ \abs{\psi(v_1, v_1)}_D \leqslant \gamma_{4em} \abs{\textnormal{disc}(L)}^{1/4em}; \]
 \item $D$ has Type $III$, $r=2$ and there exist $D$-linearly independent vectors $v_1, v_2 \in M$ such that $\psi(v_1, v_2) = 0$,
 \[ \abs{\psi(v_1, v_1)}_D, \abs{\psi(v_2, v_2)}_D \leqslant \sqrt{\gamma_{3e}} \eta^{5/3} \abs{{\textnormal{disc(R)}}}^{8e+6/e} \cdot \gamma^{8e}_{4em} \abs{\textnormal{disc(L)}}^{2/m}, \]
 and
 \[[M:Rv_1 + Rv_2] \leqslant \abs{\disc(R)}^{32e^2+24} \cdot 2^{-80e/3} \eta^{20e/3} \gamma_{3e}^e \gamma_{4em}^{32e^2-4e}  \abs{\textnormal{disc}(L)}^{(8e-1)/m}. \]
 \item  $D$ has Type $IV$,  $r=1$ and $M = Rv_1$ for some $v_1$ such that
 \[ \abs{\psi(v_1, v_1)}_D \leqslant \gamma_{2d^2em} \abs{\textnormal{disc}(L)}^{1/2d^2em}; \]
\item $D$ has Type $IV$, $r=2$ and there exist $D$-linearly independent vectors $v_1, v_2 \in M$ such that $\psi(v_1, v_2) = 0$,
 \[ \abs{\psi(v_1, v_1)}_D, \abs{\psi(v_2, v_2)}_D \leqslant 16 \sqrt{\gamma_{d^2e}} \eta^3 \abs{{\textnormal{disc(R)}}}^{4d^2e+2/e} \gamma^{4d^2e}_{2d^2em} \abs{\textnormal{disc(L)}}^{2/m}, \]
 and
 \[\begin{split}
      [M:Rv_1 + Rv_2]  \leqslant &4^{2de(1+d)} \abs{\disc(R)}^{4d^2(1+2de^2)} \eta^{6d^2e} \gamma_{d^2e}^{d^2e}  \cdot \\
  &\quad{(\gamma_{2d^2em}/(de)^{d^2e})}^{2d^2e(4d^2e-1)} \abs{\disc(L)}^{(4d^2e-1)/m}. 
 \end{split}\]
 \end{enumerate} 
\end{enumerate}
\end{lemma}

\begin{proof}
The proof relies on choosing a $D$-norm adapted to $\psi$ as in
\cref{good-norm}.  
By \citep[Prop.~2.5]{Lattices}, one can find a $D$-basis
$w_1,\dots,w_m$ contained in $L$ for which
\[
\abs{w_1}\cdots\abs{w_m}
   \le
   \gamma_{k e m}^{\,m/2}\,
   \covol(L)^{1/(ke)}
   \le
   \gamma_{k e m}^{\,m/2}\,
   \abs{\disc(L)}^{1/(2ke)},
\]
where $k=4$ (type~III) or $k=2d^2$ (type~IV).  
Selecting $(i,j)$ according to \cref{short1} gives
\begin{equation}\label{epsi1}
\abs{\psi(w_i,w_j)}_D
   \le
   \abs{w_i}\,\abs{w_j}
   \le
   \gamma_{ke m}\,\abs{\disc(L)}^{1/(ke m)}.
\end{equation}

[Proof of (i)--(iii)]. Let $M = Rw_i + Rw_j$, so that $r=1$ if $i=j$ and $r=2$ if $i \neq j$.

If $i=j$, then by \cref{short1}, $\psi(w_i, w_i) \neq 0$, so the restriction of $\psi$ to $M$ is non-degenerate.

If $i \neq j$, then by \cref{short1}, $\psi(w_i, w_i) = 0$ and $\psi(w_i, w_j) \neq 0$. Consequently, for any vector $x \in M$, if $x \in Dw_i \setminus \{0\}$, then $\psi(x, w_j) \neq 0$, while if $x \notin Dw_i$, then $\psi(x, w_i) \neq 0$. Thus the restriction of $\psi$ to $M$ is non-degenerate.

By \cref{trangular}, \cref{Normbound}(ii) and \eqref{epsi1}, we obtain that for type IV, in both cases $i=j$ or $i \neq j$,
\begin{align*}
    \abs{\text{disc}(M)}
  & = d^{-2d^2er} \abs{\text{disc}(R)}^r \abs{\Nm_{D/\mathbb{Q}}(\psi(w_i, w_j))}^r \\
  & = d^{-2d^2er} \abs{\text{disc}(R)}^r \abs{\Nrd_{D/\mathbb{Q}}(\psi(w_i, w_j))}^{dr} \\
  & \leqslant d^{-2d^2er} \abs{\text{disc}(R)}^r (de)^{-d^2er} \abs{\psi(w_i, w_j)}_D^{2d^2er} \\
  & \leqslant (d^3e)^{-d^2er} \abs{\text{disc}(R)}^r \cdot \gamma_{2d^2em}^{2d^2er} \abs{\text{disc}(L)}^{r/m}.
\end{align*}
Note that the same steps can be repeated for type III.

For the proof of (iv), we split into cases depending on whether $i=j$ or $i \neq j$.

\subsubsection*{Case~(a)}

If $D$ is a type III division $\mathbb{Q}$-algebra and $i=j$, then let $v_1 = w_i$. Then (iv)(a) holds thanks to \eqref{epsi1}.

\subsubsection*{Case~(b)}

If $D$ is a type III division $\mathbb{Q}$-algebra and $i \neq j$, then choose $\omega \in D^-$ as in \cref{as1}.
Let
\[ w_j' = 2b w_j - a w_i, \] 
where, $b=2^{2-32^{e^2}} \abs{\disc(R)}^{8e} \Nrd_{D/\mathbb{Q}}(\psi(w_i,w_j))^{4} \omega \psi(w_j,w_i)^{-1}$ \break and $a=b \psi(w_j,w_j)(\psi(w_i,w_j)^{-1})$.

By using \cref{w1w2}, $b$ and $a \in R$, so $w_j' \in Rw_i + Rw_j = M$.
Furthermore $w_j'$ and $w_i$ are $D$-linearly independent because $b \neq 0$.

Using the facts that $\psi(w_i,w_i) = 0$ and $(\psi(w_i, w_j)^{-1})^\dag = - \psi(w_j, w_i)^{-1} $, we can calculate
\begin{align*}
    \psi(w_j', w_j') & = 4b\psi(w_j, w_j)b^{\dag} -2b\psi(w_j, w_i)a^{ \dag} - 2a \psi(w_i, w_j)b^{\dag} \\
    &= 4b\psi(w_j, w_j)b^{\dag} -2b \psi(w_j, w_i) (\psi(w_i,w_j)^{-1})^{\dag} \psi(w_j,w_j)^{\dag} b^{\dag} \\ 
    & \ \ \ \ \ \ \ \ \ \ \ \ \ \ \ \ \ \ \ \ \ \  -2b\psi(w_j,w_j)(\psi(w_i,w_j)^{-1}) \psi(w_i, w_j)b^{\dag} \\
    &=(4-2-2)b\psi(w_j, w_j)b^{\dag} \\
    &=0.
\end{align*}

Now we show that  $\psi(w_j', w_i)=  \psi( w_i, w_j')$.
\begin{align}\label{153}
    \psi(w_j', w_i)
  & = 2b \psi(w_j, w_i) \\ \nonumber
   & = 2 \cdot 2^{2-32e^2} \abs{\disc(R)}^{8e} \Nrd_{D/\mathbb{Q}}(\psi(w_i,w_j))^{4} \omega (\psi (w_j,w_i))^{-1}\psi(w_j, w_i) \\ \nonumber
    &=2^{3-32e^2} \abs{\disc(R)}^{8e}  \Nrd_{D/\mathbb{Q}}(\psi(w_i,w_j))^{4} \omega .
\end{align}
And similarly, $ \psi(w_i, w_j')= 2^{3-32e^2} \abs{\disc(R)}^{8e}  \Nrd_{D/\mathbb{Q}}(\psi(w_i,w_j))^{4} \omega $.

Now let
\[ v_1 = w_i - w_j', \ \ \ \ v_2 = w_i + w_j'. \]
Clearly $v_1, v_2 \in Rw_i + Rw_j' \subset M$.
Since $w_i = \frac{1}{2}(v_1 + v_2)$ and $w_j' = \frac{1}{2}(v_2 - v_1)$, the vectors $v_1$ and $v_2$ are $D$-linearly independent.

Since $\psi(w_j', w_i) = \psi(w_i, w_j')$ we can calculate
\begin{align*}
    \psi(v_1, v_2)
  & = \psi(w_i, w_i) + \psi(w_i, w_j') - \psi(w_j', w_i) - \psi(w_j', w_j')
    = 0,
\\ \psi(v_1, v_1)
  & = \psi(w_i, w_i) - \psi(w_i, w_j') - \psi(w_j', w_i) + \psi(w_j', w_j')
    = -2\psi(w_j', w_i),
\\ \psi(v_2, v_2)
  & = \psi(w_i, w_i) + \psi(w_i, w_j') + \psi(w_j', w_i) + \psi(w_j', w_j')
    = 2\psi(w_j', w_i).
\end{align*}
Consequently using \cref{lm,as1,Normbound} and \eqref{epsi1}, \eqref{153},
\begin{align*}
 \abs{(\psi(v_1, v_1))}_D &= \abs{(\psi(v_1, v_1))}_D  =2 \abs{(\psi(w_j', w_i))}_D \\
&\leqslant 2 \cdot 2^{3-32e^2} \abs{\disc(R)}^{8e} \abs{\Nrd_{D/\mathbb{Q}}(\psi(w_i, w_j))}^{4}_D \abs{\omega}_D \\
\begin{split}
   &\leqslant 2^{4-32e^2} \abs{\disc(R)}^{8e}  \cdot ((2e)^{(-e)}\abs{(\psi(w_i, w_j))}^{2e}_D)^{4} 2^{-20/3} \eta^{5/3} \cdot \sqrt{\gamma_{3e}}  \\ 
    &\quad \quad \abs{\textnormal{disc(R)}}^{6/e} 
 \end{split}  \\
 & \leqslant 2^{4-32e^2-4e-20/3} \sqrt{\gamma_{3e}} e^{-4e} \eta^{5/3} \abs{{\textnormal{disc(R)}}}^{8e+6/e} \cdot \gamma^{8e}_{4em} \abs{\textnormal{disc(L)}}^{2/m} \\
 &\leqslant \sqrt{\gamma_{3e}} \eta^{5/3} \abs{{\textnormal{disc(R)}}}^{8e+6/e} \cdot \gamma^{8e}_{4em} \abs{\textnormal{disc(L)}}^{2/m}.
 \end{align*}
This proves the first inequality in (iv)(b).

Now by \cref{Normbound} and the fact that if $a \in \mathbb{Q}$, then $\Nrd_{D/\mathbb{Q}}(a) = (a)^{2e}$,
\begin{align*}
    \abs{\Nrd_{D/\mathbb{Q}}(\psi(v_1, v_1))}
\\  = \abs{\Nrd_{D/\mathbb{Q}}(\psi(v_2, v_2))} 
& = \abs{\Nrd_{D/\mathbb{Q}}(2(\psi(w_j',w_i))}  \\
& = \abs{\Nrd_{D/\mathbb{Q}}(2(2^{3-32e^2} \abs{\disc(R)}^{8e} \Nrd_{D/\mathbb{Q}}(\psi(w_i,w_j))^{4})) \omega} \\
&\leqslant (2^{4-32e^2} \abs{\disc(R)}^{8e})^{2e} \abs{\Nrd_{D/\mathbb{Q}}(\psi(w_i, w_j))}^{8e} \abs{\Nrd_{D/\mathbb{Q}}(\omega)} \\
&\leqslant  \abs{\disc(R)}^{16e^2} \abs{\Nrd_{D/\mathbb{Q}}(\psi(w_i, w_j))}^{8e} \abs{\Nrd_{D/\mathbb{Q}}(\omega)} 
\end{align*}
Therefore by \cref{Normbound,as1} and \eqref{epsi1},
\begin{align*}
    [M : Rv_1 + Rv_2]
  &\leqslant \frac{(\abs{\disc(R)}^{16e^2} \abs{\Nrd_{D/\mathbb{Q}}(\psi(w_i, w_j))}^{8e} \abs{\Nrd_{D/\mathbb{Q}}(\omega)})^2}{\abs{\Nrd_{D/\mathbb{Q}}(\psi(w_i, w_j))}^2}
\\& = \abs{\disc(R)}^{32e^2} \abs{\Nrd_{D/\mathbb{Q}}(\omega)}^2 \abs{\Nrd_{D/\mathbb{Q}}(\psi(w_i, w_j))}^{16e-2}
\\& \leqslant \abs{\disc(R)}^{32e^2} \cdot (2e)^{-2e} \abs{\omega}_D^{4e} \cdot ((2e)^{-e} \abs{\psi(w_i, w_j)}_D^{2e})^{16e-2}
\\&  \leqslant \abs{\disc(R)}^{32e^2} \cdot \abs{\omega}_D^{4e} \cdot (\abs{\psi(w_i, w_j)}_D)^{(32e^2-4e)}
\\& \leqslant \abs{\disc(R)}^{32e^2} \cdot 2^{-80/3e} \eta^{20/3e} \gamma_{3e}^e \abs{\disc(R)}^{24}  \gamma_{4em}^{32e^2-4e} \abs{\textnormal{disc}(L)}^{(8e-1)/m}.
\end{align*}

\subsubsection*{Case~(c)}

If $D$ is a type IV division $\mathbb{Q}$-algebra and $i=j$, then let $v_1 = w_i$. Then (iv)(c) holds thanks to \eqref{epsi1}.

\subsubsection*{Case~(d)}

If $D$ is a type IV division $\mathbb{Q}$-algebra and $i \neq j$, then choose $\omega \in D^-$ as in \cref{as1}.
Let
\[ w_j' = 2b w_j - a w_i, \] 
where, $b=d^{2-8d^4e^2} \abs{\disc(R)}^{4d^2e} \Nrd_{D/\mathbb{Q}}(\psi(w_i,w_j))^{2d} \omega \psi(w_j,w_i)^{-1}$ \break and $a=b \psi(w_j,w_j)(\psi(w_i,w_j)^{-1})$.

By using \cref{w1w2}, $b$ and $a \in R$, so $w_j' \in Rw_i + Rw_j = M$.
Furthermore $w_j'$ and $w_i$ are $D$-linearly independent because $b \neq 0$.

Using the facts that $\psi(w_i,w_i) = 0$ and $(\psi(w_i, w_j)^{-1})^\dag = - \psi(w_j, w_i)^{-1} $, we can calculate $\psi(w_j', w_j')=0$ as in Type III case.

Now we show that  $\psi(w_j', w_i)=  \psi( w_i, w_j')$.
\begin{align}\label{15}
    \psi(w_j', w_i)
  & = 2b \psi(w_j, w_i) \\ \nonumber
   & = 2 d^{2-8d^4e^2} \abs{\disc(R)}^{4d^2e} \Nrd_{D/\mathbb{Q}}(\psi(w_i,w_j))^{2d} \omega (\psi (w_j,w_i))^{-1}\psi(w_j, w_i) \\ \nonumber
    &=2 d^{2-8d^4e^2} \abs{\disc(R)}^{4d^2e}  \Nrd_{D/\mathbb{Q}}(\psi(w_i,w_j))^{2d} \omega .
\end{align}
And similarly, $ \psi(w_i, w_j')= 2 d^{2-8d^4e^2} \abs{\disc(R)}^{4d^2e}  \Nrd_{D/\mathbb{Q}}(\psi(w_i,w_j))^{2d} \omega$.


Now let
\[ v_1 = w_i - w_j', \ \ \ \ v_2 = w_i + w_j'. \]
Clearly $v_1, v_2 \in Rw_i + Rw_j' \subset M$.
Since $w_i = \frac{1}{2}(v_1 + v_2)$ and $w_j' = \frac{1}{2}(v_2 - v_1)$, the vectors $v_1$ and $v_2$ are $D$-linearly independent.

Since $\psi(w_j', w_i) = \psi(w_i, w_j')$ we can calculate $ \psi(v_1, v_2)=0, \psi(v_1, v_1)=-2\psi(w_j', w_i)=-\psi(v_2, v_2)$ as in Type III.
Consequently using \cref{lm,as1,Normbound} and \eqref{epsi1}, \eqref{15},
\begin{align*}
 \abs{(\psi(v_1, v_1))}_D &= \abs{(\psi(v_1, v_1))}_D  =2 \abs{(\psi(w_j', w_i))}_D \\
&\leqslant 4 d^{2-8d^4e^2} \abs{\disc(R)}^{4d^2e} \abs{\Nrd_{D/\mathbb{Q}}(\psi(w_i, w_j))}^{2d}_D \abs{\omega}_D \\
\begin{split}
   &\leqslant 4 d^{2-8d^4e^2} \abs{\disc(R)}^{4d^2e}  \cdot ((de)^{(-de)}\abs{(\psi(w_i, w_j))}^{2de}_D)^{2d} 4 \eta^3 \cdot \\ 
    &\quad \quad d^{-(4d^2)}\sqrt{\gamma_{d^2e}}\abs{\textnormal{disc(R)}}^{2/e} 
 \end{split}  \\
 & \leqslant 16 \sqrt{\gamma_{d^2e}} d^{2-8d^4e^2-2d^2e-4d^2} e^{-2d^2e} \eta^3 \abs{{\textnormal{disc(R)}}}^{4d^2e+2/e} \cdot \gamma^{4d^2e}_{2d^2em} \abs{\textnormal{disc(L)}}^{2/m} \\
 &\leqslant 16 \sqrt{\gamma_{d^2e}} \eta^3 \abs{{\textnormal{disc(R)}}}^{4d^2e+2/e} \cdot \gamma^{4d^2e}_{2d^2em} \abs{\textnormal{disc(L)}}^{2/m}.
 \end{align*}
This proves the first inequality in (iv)(b).

Using \cref{trangular}, we have
\begin{align*}
    [M : Rv_1 + Rv_2]
  & = \frac{\abs{\textnormal{disc}(Rv_1 + Rv_2)}^{1/2}}{\abs{\textnormal{disc}(M)}^{1/2}} = \frac{\abs{\textnormal{disc}(Rv_1 + Rv_2)}^{1/2}}{\abs{\textnormal{disc}(Rw_i + Rw_j)}^{1/2}}
\\& = \frac{\abs{\Nm_{D/\mathbb{Q}}(\psi(v_1, v_1))}^{1/2} \abs{\Nm_{D/\mathbb{Q}}(\psi(v_2, v_2))}^{1/2}}{\abs{\Nm_{D/\mathbb{Q}}(\psi(w_i, w_j))}^{1/2} \abs{\Nm_{D/\mathbb{Q}}(\psi(w_j, w_i))}^{1/2}}
\\& = \frac{\abs{\Nrd_{D/\mathbb{Q}}(\psi(v_1, v_1))}^{d/2} \abs{\Nrd_{D/\mathbb{Q}}(\psi(v_2, v_2))}^{d/2}}{\abs{\Nrd_{D/\mathbb{Q}}(\psi(w_i, w_j))}^d}.
\end{align*}
Now by \cref{Normbound} and the fact that if $a \in \mathbb{Q}$, then $\Nrd_{D/\mathbb{Q}}(a) = (a)^{2de}$,
\begin{align*}
    \abs{\Nrd_{D/\mathbb{Q}}(\psi(v_1, v_1))}
\\  = \abs{\Nrd_{D/\mathbb{Q}}(\psi(v_2, v_2))} 
& = \abs{\Nrd_{D/\mathbb{Q}}(2(\psi(w_j',w_i))}  \\
& = \abs{\Nrd_{D/\mathbb{Q}}(2(2 d^{2-8d^4e^2} \abs{\disc(R)}^{4d^2e} \Nrd_{D/\mathbb{Q}}(\psi(w_i,w_j))^{2d})) \omega} \\
&\leqslant (4d^{2-8d^4e^2} \abs{\disc(R)}^{4d^2e})^{2de} \abs{\Nrd_{D/\mathbb{Q}}(\psi(w_i, w_j))}^{4d^2e} \abs{\Nrd_{D/\mathbb{Q}}(\omega)} 
\end{align*}
Therefore by \cref{Normbound,as1} and \eqref{epsi1},
\begin{align*}
    [M : Rv_1 + Rv_2]
  & \leqslant \frac{((4d^{2-8d^4e^2} \abs{\disc(R)}^{4d^2e})^{2de}  \abs{\Nrd_{D/\mathbb{Q}}(\omega)} \abs{\Nrd_{D/\mathbb{Q}}(\psi(w_i, w_j))}^{4d^2e})^d}{\abs{\Nrd_{D/\mathbb{Q}}(\psi(w_i, w_j))}^d}
\\& = 4^{2de} d^{4de-16d^5e^3} \abs{\disc(R)}^{8d^3e^2} \abs{\Nrd_{D/\mathbb{Q}}(\omega)}^d \abs{\Nrd_{D/\mathbb{Q}}(\psi(w_i, w_j))}^{d(4d^2e-1)}
\\& \leqslant 4^{2de} d^{4de-16d^5e^3} \abs{\disc(R)}^{8d^3e^2} \cdot (de)^{-d^2e} \abs{\omega}_D^{2d^2e} \cdot ((de)^{-de} \abs{\psi(w_i, w_j)}_D^{2de})^{d(4d^2e-1)}
\\&  \leqslant 4^{2de} d^{4de-16d^5e^3} (de)^{-d^2e}\abs{\disc(R)}^{8d^3e^2} \cdot \abs{\omega}_D^{2d^2e} \cdot ((de)^{-d^2e}\abs{\psi(w_i, w_j)}_D)^{2d^2e(4d^2e-1)}
\\ \begin{split}
    &\leqslant 4^{2de} d^{4de-16d^5e^3} (de)^{-d^2e}  \abs{\disc(R)}^{8d^3e^2}  \cdot 4^{2d^2e} d^{-8d^4e}\eta^{6d^2e} \gamma_{d^2e}^{d^2e} \abs{\textnormal{disc}(R)}^{4d^2} \cdot \\
    &\quad \quad (\gamma_{2d^2em}/(de)^{d^2e})^{2d^2e(4d^2e-1)} \abs{\disc(L)}^{(4d^2e-1)/m}.
\end{split}
 \\&      \qedhere
\end{align*}

\end{proof}

\subsection{Inductive construction of a unitary basis}

We now formulate the inductive step needed to prove \cref{height-bound}.
The next result extends that theorem by allowing any
$R\subset\Stab_D(L)$ (rather than equality) and by introducing the auxiliary
constant $\eta$.  
When $R=\Stab_D(L)$, \cref{Rint} provides the needed control on $\eta$.

\begin{proposition}\label{wind}
Let $(D,\dag)$ be a type~IV division algebra over~$\mathbb{Q}$ with positive involution,
and let $V$ be a left $D$-module equipped with a nondegenerate
$(D,\dag)$-skew-Hermitian form $\psi$.  
Let $L\subset V$ be a $\mathbb{Z}$-lattice with
$\Trd_{D/\mathbb{Q}}\psi(L,L)\subset\mathbb{Z}$.  
Let $R\subset \Stab_D(L)$ be an order and choose $\eta\in\mathbb{Z}_{>0}$ such that
$\eta R^\dag\subset R$.  

Then one can find a $D$-basis $v_1,\dots,v_m$ of $V$ satisfying:
\begin{enumerate}[(i)]
\item $v_1, \dotsc, v_m \in L$;
\item the basis is weakly unitary with respect to $\psi$;
\item the index of $Rv_1 + \dotsb + Rv_m$ in $L$ is bounded as follows:
\createC{so-index-mult} \createC{so-index-eta} \createC{so-index-R} \createC{so-index-L}
\[ [L : Rv_1 + \dotsb + Rv_m] \leq \refC{so-index-mult}(d,e,m) \eta^{\refC{so-index-eta}(d,e,m)} \abs{\disc(R)}^{\refC{so-index-R}(d,e,m)} \abs{\disc(L)}^{\refC{so-index-L}(d,e,m)}; \]
\item for all $i, j \in \{ 1, \dotsc, m \}$ such that $\psi(v_i, v_j) \neq 0$,
\createC{so-psi-mult} \createC{so-psi-eta} \createC{so-psi-R} \createC{so-psi-L}
\[ \abs{\psi(v_i, v_j)}_D \leq \refC{so-psi-mult}(d, e, m) \eta^{\refC{so-psi-eta}(d,e,m)}  \abs{\disc(R)}^{\refC{so-psi-R}(d,e,m)} \abs{\disc(L)}^{\refC{so-psi-L}(d,e,m)}. \]
\end{enumerate}


The inequalities (iii) and (iv) hold with the following values of the constants:
\begin{center}
\bgroup
\renewcommand{\arraystretch}{1.4}
\begin{tabular}{c|c|c}
 & $Type ~III$
  & $Type ~IV$
\\ \hline
    $\refC{so-index-mult}(d,e,m)$
   & $(8em^2)^{5e^2m(m+2)}$
   & $(64d^3em^2)^{5d^4e^2m(m+2)}$
\\  $\refC{so-index-eta}(d,e,m)$
   &$14em$
  & $3d^2em$
\\  $\refC{so-index-R}(d,e,m)$
 &$m(m+16)/4+24m(m-1)e^2$
& $m(m+8)/4+14m(m-1)d^3e^2$
 \\  $\refC{so-index-L}(d,e,m)$
  &$(m-1)/2+4(m+1)e$
  & $(m-1)/2 + 4(m+1)d^2e$
  \\  $\refC{so-psi-mult}(d,e,m)$
  &$(4em^2)^{e((m(m+1)+14))}$
 & $(4d^2em^2)^{d^2e((m(m+1)+14))/4}$
\\  $\refC{so-psi-eta}(d,e,m)$
  & $7$
  & $3$
\\  $\refC{so-psi-R}(d,e,m)$
  & $e(m(m+1)+26)/16+12e$
  & $\bigl( em(m+1)+24d^2 \bigr)/4$
\\  $\refC{so-psi-L}(d,e,m)$
     &$(m+1)/8$
  & $(m+1)/2$
\end{tabular}
\egroup
\end{center}
\end{proposition}

\begin{proof}
    We proceed by induction on $m = \dim_D(V)$. The argument follows closely the proof of Proposition~4.5 in~\citep{Lattices}, and a more detailed version will appear in the forthcoming work of the author's thesis.
\end{proof}
\begin{lemma}[{\citep[Lemma 4.6]{Lattices}}] \label{Rint}
Let $(D,\dagger)$ be a division $\mathbb{Q}$-algebra equipped with a positive involution, and let $V$ be a left $D$-vector space endowed with a non-degenerate $(D,\dagger)$-skew-Hermitian form $\psi \colon V \times V \to D$.  
Suppose $L$ is a $\mathbb{Z}$-lattice in $V$ such that $\Trd_{D/\mathbb{Q}}\big(\psi(L \times L)\big) \subset \mathbb{Z}$.  
Let $R = \Stab_D(L)$.  
Then $\disc(L)\, R^{\dagger} \subset R$.
\end{lemma}

To conclude the proof of \cref{height-bound}, we combine \cref{wind} with \cref{Rint}.  
In part (iii), the exponent of $\lvert \disc(L) \rvert$ becomes
$\refC{so-index-eta}(d,e,m) + \refC{so-index-L}(d,e,m),$
and in part (iv), the exponent of $\lvert \disc(L) \rvert$ is
$\refC{so-psi-eta}(d,e,m) + \refC{so-psi-L}(d,e,m).$
All remaining constants in \cref{height-bound} coincide with the corresponding constants in \cref{wind}.

\end{document}